\documentclass[11pt,a4paper]{amsart}

\title[On finite parts of divergent complex geometric integrals]{On finite parts of divergent complex geometric integrals \\ and their dependence on a choice of Hermitian metric}

\author{Ludvig Svensson}
\address{Ludvig Svensson, Department of Mathematical Sciences, University of Gothenburg and 
Chalmers University of Technology, SE-412 96 G\"{o}teborg, Sweden}
\email{ludsven@chalmers.se}

\date{\today}

\usepackage[english]{babel}
\usepackage{amsmath,calligra}
\usepackage{amsthm,amssymb,latexsym}
\usepackage[all,cmtip]{xy}
\usepackage{url,ae}
\usepackage{t1enc}
\usepackage{bm}
\usepackage{mathrsfs}
\usepackage{graphicx}
\usepackage{geometry}
\usepackage{enumitem}
\usepackage{tabularx}
\usepackage{mathtools}
\usepackage[nameinlink]{cleveref}
\usepackage{xcolor}  

\crefname{figure}{Figure}{Figures}
\crefname{equation}{}{}
\crefname{table}{Table}{Tables}
\crefname{section}{§}{§§}
\crefname{appendix}{Appendix}{Appendices}

\newtheorem{proposition}{Proposition}[section]
\newtheorem{theorem}[proposition]{Theorem}
\newtheorem{lemma}[proposition]{Lemma}
\newtheorem{corollary}[proposition]{Corollary}

\theoremstyle{definition}

\newtheorem{example}[proposition]{Example}
\newtheorem{remark}[proposition]{Remark}

\numberwithin{equation}{section}

\crefname{corollary}{corollary}{corollaries}
\crefname{theorem}{theorem}{theorems}
\crefname{lemma}{lemma}{lemmas}
\crefname{definition}{definition}{definition}

\newcommand{\C}{\mathbb{C}}
\newcommand{\Q}{\mathbb{Q}}
\renewcommand{\d}{\mathrm{d}}

\def\newop#1{\expandafter\def\csname #1\endcsname{\mathop{\rm #1}\nolimits}}

\begin{document}
\nocite{*}
\bibliographystyle{plain}

\begin{abstract}

Let $X$ be a reduced complex space of pure dimension. We consider divergent integrals of certain forms on $X$ that are singular along a subvariety defined by the zero set of a holomorphic section of some holomorphic vector bundle $E \rightarrow X$. Given a choice of Hermitian metric on $E$ we define a finite part of the divergent integral. Our main result is an explicit formula for the dependence on the choice of metric of the finite part.

\end{abstract}

\maketitle
\thispagestyle{empty}

\section{Introduction}

Let $X$ be a reduced complex analytic space of pure dimension $n$ and let $V \subset X$ be an analytic subvariety. Consider an $(n,n)$-form $\omega$ which is smooth in $X \setminus V$ with singularities along $V$ and such that $\overline{\mathrm{supp}\,\omega}$ is compact in $X$. We are interested in studying finite parts of the divergent integral $\int_X \omega$, inspired by the process of regularization and renormalization in perturbative quantum field theory. In general, the finite part of a given divergent integral is not uniquely defined, rather, it depends on the choice of regularization data. It is a fundamental problem to describe this dependence.

In this paper we consider the setting when the variety $V$ is the vanishing locus of a global holomorphic section $s : X \rightarrow E$ of some holomorphic vector bundle $E \rightarrow X$. Given a (smooth) Hermitian metric $\|\cdot\|$ on $E$ we consider the space $\mathcal{A}_{s,\|\cdot\|}(X)$ of smooth differential forms $\omega$ on $X \setminus V$ such that for each compact subset $K \subset X$ there exists some integer $N\geq 0$ such that $\|s\|^{2N} \omega$ extends to a smooth form across $V \cap K$. Let $\mathcal{A}_{s}(X)$ be the union over metrics of all such $\mathcal{A}_{s,\|\cdot\|}(X)$. Note that if $s$ defines a Cartier divisor, then $|s|^2/\|s\|^2$ is smooth and non-vanishing for any two metrics $\|\cdot\|$ and $|\cdot|$ on $E$. Thus, in that case we have that $\mathcal{A}_{s,\|\cdot\|}(X) = \mathcal{A}_{s,|\cdot|}(X) = \mathcal{A}_s(X)$. In the general case $|s|^2/\|s\|^2$ is only locally bounded and there may be different \textit{conformal classes} $\mathcal{A}_{s,\|\cdot\|}(X) \subset \mathcal{A}_s(X)$.

Any $\omega\in \mathcal{A}_s^{p,q}(X)$ defines a current on $X \setminus V$, that is, a continuous linear functional on the space $\mathscr{D}^{n-p,n-q}(X\setminus V)$ of test forms on $X\setminus V$ of complementary bidegree, by
\[
    \xi \mapsto \int\limits_X \omega \wedge \xi.
\]
To find a current extension of $\omega$ across $V$, following a classical idea, we consider the function
%
\begin{equation}
    \label{eq:gamma}
    \Gamma_{\|\cdot\|}(\lambda) = \int\limits_X \|s\|^{2\lambda} \omega \wedge \xi,
\end{equation}
defined for $\mathfrak{Re}\,\lambda$ sufficiently large. Differentiation under the integral sign shows that $\Gamma_{\|\cdot\|}(\lambda)$ is holomorphic for $\mathfrak{Re}\,\lambda \gg 0$. 
It is clear that if there exists a metric $\|\cdot\|$ on $E$ such that $\|s\|^{2N}\omega$ is smooth for some $N \geq 0$, then for any other metric $|\cdot|$ on $E$, $|s|^{2N}\omega$ is locally bounded. Thus \cref{eq:gamma} is defined and holomorphic for any $\omega \in \mathcal{A}_s^{p,q}(X)$ and any choice of Hermitian metric on $E$ if $\mathfrak{Re}\,\lambda \gg 0$. It is well known that \cref{eq:gamma} has a meromorphic continuation to $\C$, see, e.g., \cite{Atiyah} and \cite{BG}. The Laurent series of $\Gamma_{\|\cdot\|}(\lambda)$ about the origin is of the form
\[
    \Gamma_{\|\cdot\|}(\lambda) = \sum\limits_{j=0}^\kappa\frac{1}{\lambda^j} \langle \mu_{j}(\omega),\xi \rangle + \mathcal{O}(\lambda),
\]
where $0\leq \kappa \leq n$ and $\mu_j(\omega)$ are currents on $X$. Moreover, $\mathrm{supp}\,\mu_j(\omega) \subseteq V$ for $j \geq 1$. See \Cref{thm:1} below for details. 
It follows that $\mu_0(\omega) = \omega$ as currents on $X\setminus V$. Thus $\mu_0(\omega)$ is a current extension of $\omega$ across $V$. For $\omega$ of top degree, and with $\mathrm{supp}\,\omega \subset\subset X$, it is therefore natural to define the finite part of $\int_X \omega$ as 
\begin{equation}
    \label{eq:finitepart}
    \mathrm{fp} \int\limits_X \omega := \langle \mu_0(\omega),1\rangle.
\end{equation}
%
This definition depends on the choice of metric on $E$, as well as on the choice of section $s$ defining $V$. In this paper our focus is the metric dependence, keeping the section fixed. In some situations, however, a change of sections can be realized a change of metrics, see \Cref{ex:3} below. The following theorem is the main result of this article. It describes the metric dependence of $\mu_j(\omega)$ for each $j = 0, \hdots, \kappa$.

%
\begin{theorem}
    \label{thm:2}
    Let $\omega \in \mathcal{A}_s^{p,q}(X)$. 
    For any two Hermitian metrics $\|\cdot\|$ and $|\cdot|$ on $E$, let $\mu_j^{\|\cdot\|}(\omega)$ and $\mu_j^{|\cdot|}(\omega)$ denote the currents defined by the coefficient of the $-j$\textsuperscript{th} order term in the Laurent series expansion around $0$ of $\Gamma_{\|\cdot\|}$ and $\Gamma_{|\cdot|}$, respectively. We have that
    \begin{equation}
        \label{eq:metric_dependence}
        \mu_j^{|\cdot|}(\omega) = \sum\limits_{\ell=0}^{n-j} \frac{1}{\ell !} \bigg( \!\log \frac{|s|^2}{\|s\|^2} \bigg)^{\!\ell}\! \mu_{j+\ell}^{\|\cdot\|}(\omega).
    \end{equation}
    %
\end{theorem}
A version of this theorem, in the special case when $X$ and $V$ are smooth, is a central result in \cite{FK1,FK2,Lthesis}, see \Cref{ex:1} below and the paragraph preceding it. There are also partial results in \cite{FK1,FK2,Lthesis} in the case when $V$ is a normal crossings divisor. 
The key idea of the proof is to consider a particular function of two complex parameters, see \cref{eq:lambdatau}, and use it to interpolate between the functions $\Gamma_{\|\cdot\|}$ and $\Gamma_{|\cdot|}$. 
\medskip

Note that the factor $\log\frac{|s|^2}{\|s\|^2}$ appearing in \cref{eq:metric_dependence} is locally integrable, but not smooth in general. This means that the products on the right-hand side of \cref{eq:metric_dependence} are not canonically defined. However, the proof of \Cref{thm:2} shows that these products have a natural meaning. 
In the special case where $s$ defines a Cartier divisor, $\log\frac{|s|^2}{\|s\|^2}$ is smooth and the products on the right hand side of \cref{eq:metric_dependence} are canonically defined. An immediate consequence of this is the following result, which generalizes some results in \cite{FK2,Lthesis}.
%
%
%
\begin{corollary}
    \label{cor:1}
    Assume that $s$ defines a Cartier divisor, and let $\kappa$ be the order of the pole of $\Gamma_{\|\cdot\|}(\lambda)$ at $0$. Then $\kappa$ and $\mu_{\kappa}^{\|\cdot\|}(\omega)$ are independent on the choice of metric.
    %
    %
    %
\end{corollary}
There is another standard way to regularize divergent integrals, such as $\int_X\omega$, which is to introduce a \emph{cut-off} parameter $\epsilon > 0$, integrate $\omega$ over the locus $\{ \|s\|^2 \geq \epsilon \}$ and then study the asymptotic behavior of the integral as $\epsilon \rightarrow 0$. For $\omega\in\mathcal{A}_s^{p,q}(X)$, $\xi\in\mathscr{D}^{n-p,n-q}(X)$, and any smooth Hermitian metric $\|\cdot\|$ on $E$, we let
\begin{equation}
    \label{eq:cutoff}
    \mathcal{I}_{\|\cdot\|}(\epsilon) = \int\limits_{\|s\|^2 \geq \epsilon} \omega \wedge \xi.
\end{equation}
The functions $\mathcal{I}_{\|\cdot\|}(\epsilon)$ and $\Gamma_{\|\cdot\|}(\lambda)$ are related via the \textit{Mellin transform}. If the limit of $\mathcal{I}_{\|\cdot\|}(\epsilon)$ as $\epsilon \rightarrow 0$ exists, we find that
\[
    \lim_{\epsilon \rightarrow 0}\mathcal{I}_{\|\cdot\|}(\epsilon) = \langle \mu^{\|\cdot\|}_0(\omega),\xi\rangle.
\]
On the other hand, if $\lim_{\epsilon\rightarrow 0}\mathcal{I}_{\|\cdot\|}(\epsilon)$ does not exist, then, using standard techniques we find that
\begin{equation}
    \label{eq:simplifiedasym}
    \mathcal{I}_{\|\cdot\|}(\epsilon) = \langle \mu_0^{\|\cdot\|}(\omega),\xi \rangle + \frac{|\log\epsilon\,|^q}{\epsilon^p} \phi(\epsilon) + \mathcal{O}(\epsilon^\delta),
\end{equation}
for some $\delta > 0$, $p,q \in \mathbb{N}$ and $\phi\in \mathscr{C}^0([0,\infty))$ such that $\phi(0) \neq 0$. Clearly $\phi$ depends on $\omega$ and $\xi$, see \Cref{thm:3} below for a more precise formula.

For $\omega$ of top degree with $\mathrm{supp}\,\omega\subset\subset X$ we have defined a finite part of $\int_X \omega$ in \cref{eq:finitepart}. Another natural definition of a finite part of $\int_X \omega$ is 
as the limit as $\epsilon \rightarrow 0$ of $\mathcal{I}_{\|\cdot\|}(\epsilon)$ (with $\xi = 1$) after having subtracted possible divergent terms. In view of \cref{eq:finitepart} and \cref{eq:simplifiedasym} we find that they are the same, that is,
\begin{equation}
    \label{eq:equality_of_finite_parts}
    \lim_{\epsilon \rightarrow 0} \bigg( \mathcal{I}_{\|\cdot\|}(\epsilon) - \frac{|\log\epsilon\,|^q}{\epsilon^p} \phi(\epsilon) \bigg) = \mathrm{fp}\int\limits_X \omega.
\end{equation}
Since the finite part extracted from $\mathcal{I}_{\|\cdot\|}(\epsilon)$ is the same as the one coming from $\Gamma_{\|\cdot\|}(\lambda)$, the metric dependence of the former thus is given by \Cref{thm:2}. Proving this metric dependence directly, without considering $\Gamma_{\|\cdot\|}(\lambda)$, seems, to the author, more involved.

\subsection{Relations to previous results}
This work is inspired by work by Felder--Kazhdan in \cite{FK1,FK2}, where the authors investigate finite parts of divergent integrals of differential forms with singularities along a submanifold $Y$ in the real setting.  
The singularities considered are determined by a \textit{conformal class of non-negative Morse--Bott functions}. These are smooth non-negative functions vanishing precisely on $Y$ with non-degenerate Hessian in the normal directions of $Y$. They consider regularizations of $\int_X \omega$ that closely resemble our $\Gamma_{\|\cdot\|}$ and $\mathcal{I}_{\|\cdot\|}$ and they investigate the dependence on the representative Morse--Bott function within a given conformal class. This is similar to the way we consider the spaces $\mathcal{A}_{s,\|\cdot\|}(X)$ and describe the metric dependence given a section $s$.

%
\begin{example}
    \label{ex:1}
    Let $X$ be a (complex) manifold, $V$ a (complex) submanifold and suppose that $s$ defines the radical ideal of $V$. Then $\kappa \leq 1$ and
    %
    \[
    \mu_0^{|\cdot|}(\omega) - \mu_0^{\|\cdot\|}(\omega) = \log\frac{|s|^2}{\|s\|^2} \mu_1(\omega),
    \]
    which is a version of a main result in \cite{FK1,FK2,Lthesis}. Note that $\mu_1(\omega)$ here is independent of the choice of metric. 

    The formula for $\mu_0^{|\cdot|}(\omega) - \mu_0^{\|\cdot\|}(\omega)$ follows directly from \Cref{thm:2} and \Cref{cor:1} if $\kappa \leq 1$. The fact that $\kappa \leq 1$ follows from \Cref{thm:1} (i) below, since $V_{\mathrm{sing}} = \varnothing = X_{\mathrm{sing}}$.

\end{example}
%
%
\begin{example}
    Suppose $X$ is a compact complex manifold and let $\omega = \alpha \wedge \bar{\beta}$, where $\alpha$ and $\beta$ are meromorphic forms of bidegree $(n,0)$, that is, locally of the form $\alpha = {f_\alpha}/{g_\alpha}$ and $\beta = {f_\beta}/{g_\beta}$ where $f_\alpha$ and $f_\beta$ are holomorphic $(n,0)$-forms and where $g_\alpha$ and $g_\beta$ are holomorphic functions. Then $\omega \in \mathcal{A}_s(X)$, where $V = \{ g_\alpha g_\beta = 0 \}$ locally. The problem of extracting a finite part of $\int_X \omega$ arises in perturbative superstring theory, see \cite[Section 7.6]{W1}. This problem is considered in \cite{FK1} in the case where $V$ is a smooth hypersurface and in \cite{FK2} when $V$ has normal crossings singularities.
\end{example}
Meromorphic functions of the form \cref{eq:gamma} also appear in a number theoretic context in \cite[Section 4]{CT}. More precisely, in \cite[Section 4]{CT} it is assumed that $E$ is a line bundle and $\omega$ is of the form $\|s\|^{-2c}\d V$, for a volume form $\d V$ on $X$, where $c$ is the corresponding integrability threshold. An explicit expression for the corresponding measure $\mu_\kappa(\omega)$ is given in \cite[Proposition 4.3]{CT}, when the divisor $D$ cut out by $s$ has simple normal crossings, expressed in terms of the Clemens complex of $D$.

\section{Preliminaries}

\subsection{Smooth forms on reduced complex analytic spaces}

We will briefly mention how one defines smooth forms on spaces with singularities, specifically reduced analytic spaces. Recall that an analytic subspace $(Z,\mathcal{O}_Z)$, or simply $Z$ when there is no risk of confusion, of a domain $\Omega \subseteq \C^n$, is a ringed space where $Z$ is given by the common vanishing locus of a collection of holomorphic functions $f_1,\hdots,f_k : \Omega \rightarrow \C$ and where the structure sheaf $\mathcal{O}_Z = \mathcal{O}_{\Omega} / \mathcal{J}_Z$, where $\mathcal{J}_Z$ is the ideal sheaf generated by $f_1,\hdots,f_k$. The space $(Z,\mathcal{O}_Z)$ is reduced if $\mathcal{J}_Z$ is radical. For $Z$ reduced, $Z_{\mathrm{reg}}$ is the set of points $z$ such that $Z$ is a manifold in a neighborhood of $z$, and $Z_{\mathrm{reg}}$ is dense in $Z$. When $Z$ is reduced, we define the sheaf $\mathscr{E}_Z$ of smooth forms on $Z$ as the quotient sheaf $\mathscr{E}_\Omega / \mathscr{N}_{Z,\Omega}$, where $\mathscr{E}_\Omega$ is the sheaf of smooth forms on $\Omega$, and $\mathscr{N}_{Z,\Omega} \subseteq \mathscr{E}_\Omega$ is the subsheaf of forms whose pullback to $Z_{\mathrm{reg}}$ vanishes.

A reduced analytic space $(X,\mathcal{O}_X)$ is a ringed space such that for any point $x \in X$ there exists a \textit{local model} consisting of an open neighborhood $U$ of $x$ and an isomorphism of ringed spaces $U \rightarrow Z$ where $Z \subset \Omega \subseteq \C^n$ is a reduced analytic subspace. The sheaf of smooth forms $\mathscr{E}_U$ on $U$, as defined above, is independent of the choice of local model. For a reduced analytic space $X$, the sheaf of smooth forms $\mathscr{E}_X$ is defined as the sheaf obtained from gluing the sheaves of smooth forms on the local models of $X$. For a more substantial treatment, see, e.g., \cite{BH,HL}.

\subsection{Currents}

On a smooth manifold $M$ of real dimension $n$, a current $\nu$ of degree $k$ is a continuous linear functional $\xi \mapsto \langle \nu,\xi\rangle$ on the space $\mathscr{D}^{n-k}(M)$ of smooth $(n-k)$-forms with compact support. 
We define the current $\d \nu$, where $\d$ is the exterior derivative, by duality; for $\xi \in \mathscr{D}^{n-k-1}(M)$ we let
\begin{equation}
    \label{eq:dofcurrent}
    \langle \d \nu, \xi \rangle := (-1)^{k+1} \langle \nu, \d \xi \rangle.
\end{equation}
Thus $\d$ takes $k$-currents to $(k+1)$-currents.

If $M$ is a complex manifold the complex structure induces a decomposition of the spaces of smooth differential $k$-forms into bigraded $(p,q)$-forms, and the exterior derivative decomposes as $\d = \partial + \bar{\partial}$. By duality, the space of $k$-currents have a similar decomposition into bigraded objects: A current of bidegree $(p,q)$ 
on $M$ acts trivially on the space $\mathscr{D}^{n-p',n-q'}(M)$ of compactly supported forms of bidegree $(n-p',n-q')$ when $(p',q') \neq (p,q)$. For a $(p,q)$-current $\nu$, we define the $(p+1,q)$ and $(p,q+1)$ currents $\partial \nu$ and $\bar{\partial}\nu$ by
\[
    \langle \partial \nu, \xi \rangle := (-1)^{p+q+1} \langle \nu, \partial \xi \rangle\quad \text{and}\quad \langle \bar{\partial} \nu, \xi \rangle := (-1)^{p+q+1} \langle \nu, \bar{\partial} \xi \rangle,
\]
respectively.

We define the \textit{support} $\mathrm{supp}\,\nu$ of a $(p,q)$-current $\nu$ as the smallest closed subset $U \subset M$ such that $\langle \nu,\xi\rangle = 0$ for each $\xi \in \mathscr{D}^{n-p,n-q}(M\setminus U)$.

For a $(p,q)$-current $\nu$ and a smooth $(p',q')$-form $\beta$, we define the $(p+p',q+q')$-current $\nu \wedge \beta$ by
\begin{equation}
    \label{eq:wedgeproduct}
    \langle \nu \wedge \beta, \xi\rangle := \langle \nu,\beta \wedge \xi \rangle,
\end{equation}
for $\xi \in \mathscr{D}^{n-p-p',n-q-q'}(M)$. We let $\beta \wedge \nu := (-1)^{(p+q)(p'+q')}\nu \wedge \beta$.

\medskip
If $X$ is a reduced analytic space, a current on $X$ is a continuous linear functional on the space $\mathscr{D}(X)$ of smooth forms with compact support. The properties of currents presented above all generalize to this setting. For a modification $f : Y \rightarrow X$ of $X$, and a current $\nu$ on $Y$, we define the \textit{push-forward} $f_* \nu$ of $\nu$ by
\begin{equation}
    \label{eq:pushforward}
    \langle f_* \nu,\xi \rangle := \langle \nu, f^* \xi \rangle,
\end{equation}
for $\xi \in \mathscr{D}(X)$. The push-forward operator is continuous and commutes with $\partial$ and $\bar{\partial}$. 
If $\beta$ is a smooth form on $X$ we have that
\begin{equation}
    \label{eq:pushforwardofproduct}
    \beta \wedge f_* \nu = f_* (f^* \beta \wedge \nu).
\end{equation}

We can generalize this product as follows: Suppose that $\beta$ is generically smooth on $X$ with $f^*\beta$ smooth on $Y$. Moreover, let $\mu$ be a current on $X$ such that $\mu = f_* \nu$ for some current $\nu$ on $Y$. Then we define
\begin{equation}
    \label{eq:generalizedproduct}
    \beta \wedge_{f,\nu} \mu := f_* (\pi^* \beta \wedge \nu).
\end{equation}

Note that, if $\beta$ is smooth, $\beta \wedge_{f,\nu} \mu = \beta \wedge \mu$ by \cref{eq:pushforwardofproduct}. Also note that the product in \cref{eq:generalizedproduct} is ill-defined in general since it depends on the choice of modification $f$ and current $\nu$. As hinted at in the introduction, products of the type \cref{eq:generalizedproduct} appear when we look at the metric dependence of the currents $\mu_j(\omega)$, cf. \Cref{thm:2} and the subsequent comments. However, as it turns out, there are canonical choices of $f$ and $\nu$ in this case, see the proof of \Cref{thm:2} below. 

The following example shows that $\beta \wedge_{f,\nu} \mu$ may be non-zero even though $\mu = 0$.
\begin{example}
    \label{ex:2}
    Consider the blowup of $\C^2$ at the origin, $\pi : \mathrm{Bl}_0 \C^2 \rightarrow \C^2$, where
    \begin{equation}
        \label{eq:bl0c2}
        \mathrm{Bl}_0 \C^2 = \big\{(z_1,z_2,[w_0 : w_1]) \in \C^2_z \times \mathbb{P}^1_{[w]}: z_1 w_1 - z_2 w_0 = 0 \big\},
    \end{equation}
    and $\pi$ is the restriction of the natural projection $\Pi : \C^2 \times \mathbb{P}^1 \rightarrow \C^2$ to $\mathrm{Bl}_0\C^2$. Let
    \[
        \beta = \frac{i}{2\pi} \partial \bar{\partial}\log(|z_1|^2 + |z_2|^2).
    \]
    Then $\widetilde{\beta} = \pi^* \beta = \omega_{FS}(w)|_{\mathrm{Bl}_0 \C^2}$, that is, the Fubini--Study form on $\mathbb{P}^1_{[w]}$, extended to $\C^2_z \times \mathbb{P}^1_{[w]}$ and restricted to $\mathrm{Bl_0} \C^2$.
    
    One way to see this is as follows: Away from the origin, $\pi$ is a biholomorphism, so $\beta$ and $\widetilde{\beta}$ are related via a holomorphic change of coordinates. The Fubini--Study form on $\mathbb{P}^1$ with homogeneous coordinates $[w_0:w_1]$ is given by
    \[
        \omega_{FS} = \frac{i}{2\pi}\partial \bar{\partial} \log (|w_0|^2 + |w_1|^2).
    \]
    Away from $z_1 = 0$ and $[w_0 : w_1] = [0:1]$ we see from \cref{eq:bl0c2} that
    \[
        \frac{z_2}{z_1} = \frac{w_1}{w_0}.
    \]
    Since $\partial\bar{\partial}\log|g|^2 = 0$ if $g$ is holomorphic and non-vanishing it follows that
    \begin{align*}
        \omega_{FS} &= \frac{i}{2\pi}\partial\bar{\partial}\log(|w_0|^2 + |w_1|^2) 
        %
        %
        %
        = \frac{i}{2\pi}\partial\bar{\partial}\log(|z_1|^2 + |z_2|^2).
    \end{align*}
    By a symmetrical argument, $\omega_{FS} = \frac{i}{2\pi}\partial\bar{\partial}\log(|z_1|^2 + |z_2|^2)$ away from $z_2 = 0$ and $[w_0 : w_1] = [1 : 0]$. 
    
    Now, let $\nu = [E]$ be the integration current for the exceptional divisor $E=\pi^{-1}(0)$ on $\mathrm{Bl}_0 \C^2$. Then, e.g., by the dimension principle, $\mu := \pi_* \nu = 0$. However, for $\xi \in \mathscr{D}^{0,0}(\C^2)$, we have by \cref{eq:pushforward} and \cref{eq:generalizedproduct} that
    \begin{align*}
        \langle \beta \wedge_{\pi,\nu} \mu, \xi \rangle &= \langle \widetilde{\beta} \wedge [E], \pi^* \xi \rangle 
        = \int\limits_E \omega_{FS} \, \pi^* \xi 
        = \xi(0) \int\limits_E \omega_{FS} 
        = \xi(0).
    \end{align*}
    Thus, we conclude that $\varphi \wedge_{\pi,\nu} \mu = \delta_0$, where $\delta_0$ is the Dirac distribution. 
\end{example}
%

\section{Meromorphic continuation}
In this section we show the existence of meromorphic continuations of functions, closely related to \cref{eq:gamma}, which we will make use of in the proof of \Cref{thm:2}. Recall that $X$ is a reduced analytic space, $E \rightarrow X$ is a holomorphic vector bundle, and $s$ is a holomorphic section of $E$ with $V = \{s=0\}$. 
\begin{proposition}
    \label{prop:A}
    Let $\omega \in \mathcal{A}_s^{n,n}(X)$ with $\mathrm{supp}\,\omega \subset\subset X$, and let $|\cdot|$ and $\|\cdot\|$ be two Hermitian metrics on $E$. Then the function
    \begin{equation}
        \label{eq:lambdatau}
        (\lambda,\tau) \mapsto \int\limits_{X} \|s\|^{2\lambda} \left( \frac{|s|}{\|s\|}\right)^{\!\!2\tau}\!\! \omega,
    \end{equation}
    a priori defined and holomorphic for $\mathfrak{Re}\,\lambda \gg 0$, has a meromorphic continuation to $\C^2$, and there is a discrete subset $P \subset \Q \cap (-\infty,N]$, for some $N \geq 0$, such that the polar locus $\subseteq P \times \C_\tau$.
\end{proposition}
One can show \Cref{prop:A} using Berstein--Sato theory in a standard way, see, e.g., \cite{BE1,BE2,BG}. We choose here instead to use Hironaka's theorem to reduce the proof to an elementary calculation. This approach is common in residue calculus, see, e.g., \cite{A1}.

Note that $|s|^2/\|s\|^2$ is not only locally bounded but everywhere positive. Thus we can find constants $C_1,C_2 > 0$ such that $C_1 < |s|^2/\|s\|^2 < C_2$ on $\overline{\mathrm{supp}\,\omega}$. This implies that \cref{eq:lambdatau} is defined and holomorphic for any $\tau \in \C$ provided that $\mathfrak{Re}\,\lambda \gg 0$.


Our proof of \Cref{prop:A} relies on the following lemma.
\begin{lemma}
    \label{lem:i}
    Let $\Psi$ be a smooth compactly supported function on $\C^n_z$, let $v$ and $w$ be smooth positive functions defined in a neighborhood of $\mathrm{supp}\,\Psi$, let $1\leq \kappa \leq n$ and let $m_1,\hdots,m_\kappa$ be positive integers. Then, for any non-negative integer $N$, the function
    \[
        \Gamma(\tau,\lambda) = \int\limits_{\C^n} |z^{m_1}_1 \cdots z^{m_\kappa}_\kappa|^{2(\lambda - N)} v^\lambda w^\tau \Psi \d z \wedge \d \bar{z},
    \]
    where $\d z \wedge \d \bar{z} = \d z_1 \wedge \bar{z}_1 \wedge \cdots \wedge \d z_n \wedge \d \bar{z}_n$, is holomorphic for $\mathfrak{Re}\,\lambda \gg 0$, and has a meromorphic continuation to $\C^2$. Moreover, there is a discrete subset $P \subset \Q \cap (-\infty,N]$ such that the polar locus is contained in $P \times \C_\tau$, $\forall \Psi,v,w$.
\end{lemma}
A computation similar to the following proof can be found in \cite{Lthesis}. We provide our adapted version for future reference. 
\begin{proof}[Proof of \Cref{lem:i}]
    For $\mathfrak{Re}\,\lambda \gg 0$, we have that
    \[
        \frac{\partial^2}{\partial z_1 \partial \bar{z}_1}|z_1^{m_1}|^{2\lambda} = m_1^2 \lambda^2 \frac{|z_1^{m_1}|^{2\lambda}}{|z_1|^2}.
    \]
    By an induction argument it follows that
    \begin{equation}
        \label{eq:derivatives}
        |z_1^{m_1} \cdots z_\kappa^{m_\kappa}|^{2(\lambda-N)} = \frac{h(\lambda)}{\lambda^{2\kappa}}\frac{\partial^{2 N \sum_{j=1}^\kappa m_j}}{\partial z_1^{N m_1} \partial \bar{z}_1^{N m_1} \cdots \partial z_\kappa^{N m_\kappa} \partial \bar{z}_\kappa^{N m_\kappa}} |z_1^{m_1} \cdots z_\kappa^{m_\kappa}|^{2\lambda},
    \end{equation}
    where
    \begin{equation}
        \label{eq:hlambda}
        h(\lambda) = \prod\limits_{i=1}^\kappa \frac{1}{m_i^2} \prod\limits_{j=1}^{N m_i -1} \frac{1}{(m_i\lambda - j)^2}.
    \end{equation}
    By writing
    \[
         \frac{\partial^{2 N \sum_{i=1}^\kappa m_i}}{\partial z_1^{N m_1} \partial \bar{z}_1^{N m_1} \cdots \partial z_\kappa^{N m_\kappa} \partial \bar{z}_\kappa^{N m_\kappa}} = P \bar{P} \quad \text{where}\quad P = \frac{\partial^{N \sum_{i=1}^\kappa m_i}}{\partial z_1^{N m_1} \cdots \partial z_\kappa^{N m_\kappa}},
    \]
    \cref{eq:derivatives} then becomes
    \begin{equation}
        \label{eq:derivatives2}
        |z_1^{m_1} \cdots z_\kappa^{m_\kappa}|^{2(\lambda-N)} = \frac{h(\lambda)}{\lambda^{2\kappa}} P \bar{P} |z_1^{m_1} \cdots z_\kappa^{m_\kappa}|^{2\lambda}.
    \end{equation}
    Using \cref{eq:derivatives2} and integration by parts, and the fact that $(P\bar{P})^* = P\bar{P}$, we find that
    \begin{equation}
        \label{eq:merocont}
        \Gamma(\lambda,\tau) = \frac{h(\lambda)}{\lambda^{2\kappa}} \int\limits_{\C^n} |z_1^{m_1} \cdots z_\kappa^{m_\kappa}|^{2\lambda} P \bar{P}(v^\lambda w^\tau \Psi) \d z \wedge \d \bar{z},
    \end{equation}
    for $\mathfrak{Re}\,\lambda \gg 0$. The integral on the right-hand side of \cref{eq:merocont} is an analytic function of $(\lambda,\tau)$ for $\mathfrak{Re}\,\lambda > - \epsilon$ for a small enough $\epsilon > 0$, and $h(\lambda)$ is a meromorphic function on $\C_\lambda$ with poles at
    \[
        \lambda = \frac{1}{m_i}, \frac{2}{m_i}, \cdots, \frac{N m_i-1}{m_i},\quad i = 1,\hdots, \kappa.
    \]
    It follows that $\Gamma(\lambda,\tau)$ can be meromorphically continued to $\{\mathfrak{Re}\,\lambda > -\epsilon\} \times \C_\tau$. 
    
    For any integer $M \geq 0$ and $\mathfrak{Re}\,\lambda \gg 0$, it follows from \cref{eq:derivatives} by changing $\lambda$ to $\lambda + M$ and $N$ to $N+M$ that
    \begin{equation}
        \label{eq:merocontM1}
        |z^{m_1}_1 \cdots z^{m_\kappa}_\kappa|^{2(\lambda - N)} = \frac{h_M(\lambda)}{\lambda^{2\kappa}} P_M \bar{P}_M |z^{m_1}_1 \cdots z^{m_\kappa}_\kappa|^{2(\lambda+M)},
    \end{equation}
    where
    \begin{equation}
        \label{eq:hMlambda}
        h_M(\lambda) = \prod\limits_{i=1}^{\kappa} \frac{1}{m_i^2} \prod\limits_{j=1}^{(N+M)m_i-1} \frac{1}{(m_i (\lambda + M) - j)^2},
    \end{equation}
    and
    \[
        P_M = \frac{\partial^{(N+M)\sum_{i=1}^\kappa m_i}}{\partial z_1^{(N+M)m_1} \cdots \partial z_\kappa^{(N+M)m_\kappa}}.
    \]
    Analogously to \cref{eq:merocont} we then have that
    \begin{equation}
        \label{eq:merocontM2}
        \Gamma(\lambda,\tau) = \frac{h_M(\lambda)}{\lambda^{2\kappa}} \int\limits_{\C^n} |z_1^{m_1} \cdots z_\kappa^{m_\kappa}|^{2(\lambda+M)} P_M \bar{P}_M (v^\lambda w^\tau \Psi) \d z \wedge \d \bar{z},
    \end{equation}
    for $\mathfrak{Re}\,\lambda \gg 0$. The integral on the right-hand side of \cref{eq:merocontM2} is an analytic function of $(\lambda,\tau)$, now for $\mathfrak{Re}\,\lambda > -M - \epsilon$, and $h_M(\lambda)$ is a meromorphic function on $\C_\lambda$ with poles at
    \[
        \lambda = \frac{1}{m_i} - M,\frac{2}{m_i} - M,\hdots,\frac{(N + M)m_i -1}{m_i} - M,\quad i = 1,\hdots,\kappa.
    \]
    Since $M$ is arbitrary, it follows that $\Gamma(\lambda,\tau)$ has a meromorphic continuation to $\C^2$. We also see that there is a discrete subset $P \subseteq \Q \cap (-\infty,N]$, such that the polar locus of $\Gamma(\lambda,\tau)$ is contained in $P \times \C_\tau$, independent of $v,w$ and $\Psi$.
\end{proof}
\begin{proof}[Proof of \Cref{prop:A}]
    We note that we can find constants $C_1,C_2 > 0$ such that $C_1 < |s|^2 / \|s\|^2 < C_2$ on $\overline{\mathrm{supp}\,\omega}$, and that \cref{eq:lambdatau} is analytic for $\mathfrak{Re}\,\lambda$ sufficiently large. Let $\pi : \widetilde{X} \rightarrow X$ be a modification such that $\widetilde{X}$ is smooth and $\pi^* s$ defines a normal crossings divisor on $\widetilde{X}$. Since $\pi$ is a biholomorphism outside a set of measure $0$ we have, for $\mathfrak{Re}\,\lambda \gg 0$,
    \begin{equation}
        \label{eq:modification}
        \int\limits_X \|s\|^{2\lambda} \left( \frac{|s|}{\|s\|}\right)^{\!\!2\tau}\!\! \omega = \int\limits_{\widetilde{X}} \|\pi^* s\|^{2\lambda} \left( \frac{|\pi^*s|}{\|\pi^*s\|}\right)^{\!\!2\tau}\!\! \pi^*\omega. 
    \end{equation}
    We can find an open cover $\{ U_j \}$ such that, in each $U_j$, there are local holomorphic coordinates $z = (z_1,\hdots,z_n)$ such that either $\pi^* \omega$ is smooth or there is some $1 \leq \kappa \leq n$ such that $\| \pi^* s \|^2 = |z_1^{m_1} \cdots z_\kappa^{m_\kappa}|^2 e^{-\phi}$ and $| \pi^* s |^2 = |z_1^{m_1} \cdots z_\kappa^{m_\kappa}|^2 e^{-\psi}$ for some $m_1,\hdots,m_\kappa \geq 1$ and $\phi,\psi \in \mathscr{C}^\infty(U_j,\mathbb{R})$. It follows that $\|\pi^* s \|^{2\lambda} \pi^*\omega$ is smooth for $\mathfrak{Re}\,\lambda$ sufficiently large. Thus, we can find an integer $N \geq 0$ such that
    \begin{equation}
        \label{eq:pullbackomega}
        \pi^*\omega = \frac{\Psi \d z \wedge \d \bar{z}}{|z_1^{m_1} \cdots z_\kappa^{m_\kappa}|^{2N}},
    \end{equation}
    where $\Psi$ is a smooth function. By introducing a partition of unity $(\rho_j)$ subordinate to $\{U_j\}$ we find that the right-hand side of \cref{eq:modification} is a finite sum of terms of the form
    \[
        \int\limits_{\C^n} |z_1^{m_1}\cdots z_\kappa^{m_\kappa}|^{2(\lambda - N)} e^{-\lambda \phi} e^{-\tau(\psi - \phi)} \rho_j \Psi \d z \wedge \d \bar{z}.
    \]
    Note that the constants $\kappa,m_1,\hdots,m_\kappa$ depend on the local chart $U_j$, although we have suppressed this dependence in the notation. The proof now follows by \Cref{lem:i}, with $v = e^{-\phi}$ and $w = e^{-(\psi-\phi)}$. By the uniqueness of meromorphic continuation, it is independent of the particular choice of modification.
\end{proof}
%

\section{The currents $\mu_j(\omega)$ associated with $\Gamma_{\|\cdot\|}$}

In this section we prove the following theorem, which is a collection of know results together with applications of classical ideas, see, e.g., \cite{Atiyah,BA1,BE1,BE2,BG,FK1,FK2,Lthesis}, and also, e.g., \cite{A1, BGVY, PT,Yger} and references therein for analogous results in residue theory. 
We supply details of the proof for completeness, and to gather and organize these results and techniques in our setting.
\begin{theorem}
    \label{thm:1}
    Let $\omega \in \mathcal{A}_s^{p,q}(X)$, $\xi \in \mathscr{D}^{n-p,n-q}(X)$ and let $\|\cdot\|$ be a Hermitian metric on $E$.
    \begin{enumerate}[label = (\roman*)]
        \item The function
        \begin{equation*}
            \Gamma_{\|\cdot\|}(\lambda) = \int\limits_X \|s\|^{2\lambda} \omega \wedge \xi,
        \end{equation*}
        a priori defined and holomorphic for $\mathfrak{Re}\,\lambda \gg 0$, extends to a meromorphic function on $\C$ with polar set contained in $\Q$. Moreover, there exists a $\kappa \leq n$ 
        such that the Laurent series expansion of $\Gamma_{\|\cdot\|}$ in a neighborhood of $0$ is given by
        \begin{equation}
            \label{eq:gammalaurent}
            \sum\limits_{j = 0}^\kappa \frac{1}{\lambda^j} \langle \mu_j(\omega),\xi\rangle + \mathcal{O}(\lambda),
        \end{equation}
        where $\mu_j(\omega)$ are currents on $X$ satisfying $\mathrm{supp}\,\mu_\kappa(\omega) \subseteq \mathrm{supp}\,\mu_{\kappa-1}(\omega) \subseteq \cdots \subseteq \mathrm{supp}\,\mu_1(\omega) \subseteq \mathrm{supp}\,\mu_0(\omega) = \overline{\mathrm{supp}\,\omega}$. Moreover, $\mathrm{supp}\,\mu_1(\omega) \subseteq V$ and if $s$ defines the radical ideal of $V$ then $\mathrm{supp}\,\mu_j(\omega) \subseteq V_{\mathrm{sing}} \cup (X_{\mathrm{sing}}\cap V)$ for $j\geq 2$.
        \item Suppose that $\omega \in \mathcal{A}_{s,\|\cdot\|}(X)$. Then we have that $\d \omega, \frac{\d\|s\|^2}{\|s\|^2} \wedge \omega \in \mathcal{A}_{s,\|\cdot\|}(X)$, and, for any $j$,
        \begin{equation}
            \label{eq:dmu}
            \d \mu_j(\omega) = \mu_j(\d \omega) + \mu_{j+1} \bigg( \frac{\d \|s\|^2}{\|s\|^2} \wedge \omega \bigg).
        \end{equation}
    \end{enumerate}
\end{theorem}
%
%
%

\medskip

\subsection{Proof of \Cref{thm:1} (i)}

To begin with we consider the case where $\omega$ is of top degree and $s$ defines a normal crossings divisor. We have the following lemma.
\begin{lemma}
    \label{lem:ii}
    Suppose that $X$ is a manifold and that $s$ defines a normal crossings divisor with support $V = \{s=0\}$. Let $\omega \in \mathcal{A}_s^{n,n}(X)$ and $\|\cdot\|$ be any Hermitian metric on $E$. For any test function $\xi \in \mathscr{D}^{0,0}(X)$ we let
    \[
        \Gamma_{\|\cdot\|}(\lambda) = \int\limits_X \|s\|^{2\lambda} \omega \xi,
    \]
    for $\mathfrak{Re}\,\lambda \gg 0$. Then $\Gamma_{\|\cdot\|}$ has a meromorphic continuation to $\C_\lambda$ with polar set given by a discrete subset $P \subset \Q \cap (-\infty, N]$ for some $N \geq 0$ independent of $\|\cdot\|$ and $\xi$. Moreover, there exists a $0 \leq \kappa \leq n$ such that the Laurent series expansion of $\Gamma_{\|\cdot\|}$ around $0$ is given by
    \[
        \Gamma_{\|\cdot\|}(\lambda) = \sum\limits_{j=0}^\kappa \frac{1}{\lambda^j} \langle \mu_j(\omega), \xi\rangle + \mathcal{O}(\lambda),
    \]
    where $\mu_j(\omega)$, for $j=0,\hdots,\kappa$, are $(n,n)$-currents on $X$ satisfying $\mathrm{supp}\,\mu_\kappa(\omega) \subseteq \mathrm{supp}\,\mu_{\kappa-1}(\omega) \subseteq \cdots \subseteq \mathrm{supp}\,\mu_1(\omega) \subseteq \mathrm{supp}\,\mu_0(\omega) = \overline{\mathrm{supp}\,\omega}$, $\mathrm{supp}\,\mu_1(\omega) \subseteq V$ and $\mathrm{supp}\,\mu_j(\omega) \subseteq V_{\mathrm{sing}}$ for $j\geq 2$.
\end{lemma}
\begin{proof}
    The statement that $\Gamma_{\|\cdot\|}$ has a meromorphic continuation with the prescribed polar set follows immediately from \Cref{prop:A} by setting $\tau = 0$.
    %
    Now, consider the Laurent series expansion of $\Gamma_{\|\cdot\|}(\lambda)$ around $\lambda = 0$,
    \[
        \Gamma_{\|\cdot\|}(\lambda) = \sum\limits_{j=0}^{N_0} \frac{1}{\lambda^j} c_j + \mathcal{O}(\lambda),
    \]
    where $c_j \in \C$, for some $N_0 \geq 0$. Since being a current is a local property, we may assume that $\xi$ has support in some neighborhood where we can find local holomorphic coordinates $z = (z_1,\hdots,z_n)$ such that $\|s\|^{2} = |z_1^{m_1}\cdots z_\kappa^{m_\kappa}|^2 e^{-\phi}$ for some $1\leq \kappa \leq n$ and $\phi \in \mathscr{C}^\infty(\mathrm{supp}\,\xi,\mathbb{R})$. Since $\|s\|^{2\lambda} \omega$ is smooth for $\mathfrak{Re}\,\lambda$ sufficiently large, we can find an integer $N \geq 0$ and a smooth function $\Psi$ such that $\omega$ is given by the right-hand side of \cref{eq:pullbackomega}. Thus, in a neighborhood of $\lambda = 0$ we know from the proof of \Cref{lem:i}, cf. \cref{eq:merocont}, that we may write
    \[
        \Gamma_{\|\cdot\|}(\lambda) = \frac{h(\lambda)}{\lambda^{2\kappa}} I(\lambda),
    \]
    where $h(\lambda)$ is given by \cref{eq:hlambda}, and where
    \begin{equation}
        \label{eq:Ilambda}
        I(\lambda) = \int\limits_{X} |z_1^{m_1} \cdots z_\kappa^{m_\kappa}|^{2\lambda} P \bar{P} \big( e^{-\lambda\phi}\Psi \xi  \big)\d z \wedge \d \bar{z},
    \end{equation}
    with $P$ as in the proof of \Cref{lem:i}. In particular, both $h$, and $I$ are holomorphic in a neighborhood of $0$. We have that
    \begin{align}
        \nonumber
        c_j &= \underset{\lambda = 0}{\mathrm{Res}} \left\{ \lambda^{j-1} \Gamma_{\|\cdot\|}(\lambda) \right\} = \underset{\lambda = 0}{\mathrm{Res}} \left\{ \frac{1}{\lambda^{2\kappa - j + 1}} h(\lambda) I(\lambda) \right\} \\
        \nonumber
        &= \frac{1}{(2\kappa-j)!} \frac{\d^{2\kappa-j}}{\d \lambda^{2\kappa-j}} \big( h(\lambda) I(\lambda) \big)\Big|_{\lambda = 0} \\
        \label{eq:cjexpansion}
        &= \frac{1}{(2\kappa-j)!} \sum\limits_{\ell=0}^{2\kappa-j} \binom{2\kappa-j}{\ell} h^{(\ell)}(0) I^{(2\kappa-j-\ell)}(0).
    \end{align}
    Let $k = 2\kappa-j-\ell$, and consider $I^{(k)}(0)$. A standard computation with Leibniz rule gives that
    \begin{align*}
        I^{(k)}(0) &= \frac{\d^k }{\d \lambda^k}I(\lambda) \big|_{\lambda = 0} \\
        &= \sum\limits_{r=0}^{k} \binom{k}{r} (-1)^r \int\limits_{\C^n} \big(\log | z_1^{m_1} \cdots z_\kappa^{m_\kappa}|^{2}\big)^{k-r} P\bar{P}\big( \phi^r \Psi \xi \big)\d z \wedge \d \bar{z} \\
        &= \sum\limits_{r=0}^{k} \binom{k}{r} (-1)^r \int\limits_{\C^n} \big(\log | z_1^{m_1} |^2 + \cdots + \log |z_\kappa^{m_\kappa}|^{2}\big)^{k-r} P\bar{P}\big( \phi^r \Psi \xi \big)\d z \wedge \d \bar{z}.
    \end{align*}
    By the multinomial theorem we have that
    \[
        \big(\log | z_1^{m_1} |^2 + \cdots + \log |z_\kappa^{m_\kappa}|^{2}\big)^{k-r} = \;\sum_{\mathclap{\substack{\alpha \in \mathbb{Z}^{\kappa}_{\geq 0} \\ |\alpha| = k-r}}} \;\frac{(k-r)!}{\alpha_1! \cdots \alpha_\kappa!}\prod_{t = 1}^\kappa \big(\log|z_t^{m_t}|^2\big)^{\alpha_t}.
    \]
    We see that if $k-r < \kappa$, each multi-index $\alpha$ will contain at least one $0$ entry. Suppose, for simplicity, that $\alpha_1 = 0$ for a given term. Then we clearly have that
    \[
        \frac{\partial}{\partial z_1} \prod_{t = 1}^\kappa \big(\log|z_t^{m_t}|^2\big)^{\alpha_t} = \frac{\partial}{\partial z_1} \prod_{t = 2}^\kappa \big(\log|z_t^{m_t}|^2\big)^{\alpha_t} = 0.
    \]
    If $k-r<\kappa$, it follows that
    \[
        \int\limits_{\C^n} \big(\log | z_1^{m_1} |^2 + \cdots + \log |z_\kappa^{m_\kappa}|^{2}\big)^{k-r}P\bar{P}\big( \phi^r \Psi \xi \big)\d z \wedge \d \bar{z} = 0
    \]
    by integration by parts; thus, $I^{(k)}(0) = 0$ if $k<\kappa$. From \cref{eq:cjexpansion} it follows that $c_j = 0$ if $2\kappa - j < \kappa$, that is, for $j > \kappa$. Thus, we have that
    \[
        \Gamma_{\|\cdot\|}(\lambda) = \sum\limits_{j=1}^{\kappa} \frac{1}{\lambda^j} c_j + \mathcal{O}(\lambda).
    \]
    We see from \cref{eq:cjexpansion} and the expansion of $I^{(k)}(0)$ that $c_j$, for each $j=0,\hdots,\kappa$, consists of a finite sum of integrals of the form
    \[
        \int\limits_{\C^n} \big(\log | z_1^{m_1} |^2 + \cdots + \log |z_\kappa^{m_\kappa}|^{2}\big)^{k-r}P \bar{P}\big(\phi^r \Psi \xi\big) \d z \wedge \d \bar{z}.
    \]
    %
    %
    %
    Since $(\log|z_1^{m_1}|^2 + \cdots + \log|z_\kappa^{m_\kappa}|^2)^{k-r}$ is locally integrable in $\C^n$, it follows by the product rule that $c_j$ consists of a finite sum of integrals, where the integrands consist of derivatives on the test function $\xi$ multiplied by $L^1_{\mathrm{loc}}$-functions. This immediately implies that $c_j$ defines the action of a $(n,n)$-current on $\xi$, which we denote by $c_j = \langle \mu_j(\omega), \xi \rangle$.

    In \cite{Lthesis} it is shown that $\mathrm{supp}\,\mu_j(\omega) \subseteq \mathrm{supp}\,\mu_{j-1}(\omega)$ for each $j=1,\hdots,\kappa$. For convenience we sketch an argument. 
    Let
    \begin{equation}
        \label{eq:Ikr}
        I_{k, r} = \;\sum_{\mathclap{\substack{\alpha \in \mathbb{Z}^{\kappa}_{\geq 0} \\ |\alpha| = k-r}}} \;\frac{(k-r)!}{\alpha_1! \cdots \alpha_\kappa!} \int\limits_{\C^n} \prod_{t = 1}^\kappa \big(\log|z_t^{m_t}|^2\big)^{\alpha_t} P\bar{P} \big(\phi^r \Psi \xi \big) \d z \wedge \d \bar{z}.
    \end{equation}
    Then
    \begin{equation}
        \label{eq:Ik}
        I^{(k)}(0) = \sum\limits_{r=0}^k \binom{k}{r}(-1)^r I_{k, r}.
    \end{equation}
    %
    By the above, we know that $I_{k, r} = 0$ if $k-r < \kappa$. If $k-r = \kappa$, then partial integration shows that each term in the right-hand side of \cref{eq:Ikr} except for the $\alpha = (1,\hdots,1)$ term vanishes. It follows that
    \begin{align*}
        I_{k, r} &= \kappa! \int\limits_{\C^n} \prod_{t = 1}^\kappa m_t \frac{\partial^2}{\partial z_t \partial \bar{z}_t}\big(\log|z_t|^2 \big)\frac{\partial^{2 N \sum_{i=1}^\kappa m_i - 2\kappa}}{\partial z_1^{N m_1 -1} \cdots \partial \bar{z}_\kappa^{N m_\kappa -1}} \big( \phi^r \Psi \xi \big) \d z \wedge \d \bar{z},
    \end{align*}
    where $\frac{\partial^2}{\partial z_t \partial \bar{z}_t} \log|z_t|^2$ is to be regarded as a distribution. Thus, if $k-r = \kappa$, by repeated use of the Poincaré--Lelong formula,
    \[
        \partial \bar{\partial} \log|z_j|^2 = - 2\pi i [z_j = 0],
    \]
    for $j=1,\hdots,\kappa$, we have that
    \[
        I_{k, r} = \kappa! (-2\pi i)^\kappa \prod_{t = 1}^\kappa m_t  \!\!\!\! \int\limits_{\{z_1=\cdots =z_\kappa = 0\}} \!\!\!\!\frac{\partial^{2 N \sum_{i=1}^\kappa m_i - 2\kappa}}{\partial z_1^{N m_1 -1} \cdots \partial \bar{z}_\kappa^{N m_\kappa -1}} \big( \phi^r \Psi \xi \big) \d z' \wedge \d \bar{z}',
    \]
    where $\d z' \wedge \d \bar{z}' = \d z_{\kappa + 1} \wedge \d \bar{z}_{\kappa + 1} \wedge \cdots \wedge \d z_n \wedge \d \bar{z}_n$. 
    By \cref{eq:cjexpansion} and \cref{eq:Ik} it follows that
    \[
        c_\kappa = h(0) I_{\kappa,0},
    \]
    whence $\mathrm{supp}\,\mu_{\kappa}(\omega) \subseteq \{z_1=\cdots=z_\kappa = 0\} \cap \mathrm{supp}\,\Psi$. Similarly, by \cref{eq:cjexpansion} and \cref{eq:Ik} it follows that
    \[
        c_{\kappa-1} = \frac{1}{(\kappa+1)!}h(0) I_{\kappa+1,0} - \frac{1}{\kappa!} h(0) I_{\kappa+1,1} + \frac{1}{\kappa!} h'(0) I_{\kappa,0}.
    \]
    From \cref{eq:Ikr}, setting $k = \kappa + 1$ and $r = 1$, we find that only the term with $\alpha=(1,\hdots,1)$ gives a non-zero contribution to $I_{\kappa+1,1}$. Thus, by the same argument as above, $I_{\kappa+1,1}$ is an integral over the locus $\{z_1=\cdots = z_\kappa=0\}$. Looking at the expression for $I_{\kappa+1,0}$, we find that the only terms that contribute are $\alpha = (2,1,\hdots,1), (1,2,1,\hdots,1), \hdots,$ $(1,\hdots,1,2)$. Consider for example the term with $\alpha=(2,1,\hdots,1)$,
    %
    \begin{align*}
        (I_{\kappa+1, 0})_\alpha &= \frac{(\kappa+1)!}{2!} \int\limits_{\C^n} \big(\log|z_1^{m_1}|^2\big)^{2} \prod_{t = 2}^\kappa \log|z_t^{m_t}|^2 \frac{\partial^{2N \sum_{i=1}^\kappa m_i}}{\partial z_1^{N m_1} \cdots \partial \bar{z}_\kappa^{N m_\kappa}} \big(\Psi \xi \big) \d z \wedge \d \bar{z} \\
        &= \frac{(\kappa+1)!}{2!} \int\limits_{\C^n} \big(\log|z_1^{m_1}|^2\big)^2 \prod_{t = 2}^\kappa \frac{\partial^2}{\partial z_t \partial \bar{z}_t} \big( \log|z_t^{m_t}|^2\big) \times \\
        &\qquad \qquad \qquad \qquad \ \times \frac{\partial^{2N \sum_{i=1}^\kappa m_i - 2(\kappa-1)}}{\partial z_1^{N m_1} \partial \bar{z}_1^{N m_1} \partial z_2^{N m_2 - 1} \cdots \partial \bar{z}_\kappa^{N m_\kappa - 1}} \big(\Psi \xi \big) \d z \wedge \d \bar{z} \\
        &= \frac{(\kappa+1)!}{2!}(-2\pi i)^{\kappa-1} \prod\limits_{t=1}^\kappa m_t \!\!\!\! \int\limits_{\{z_2=\cdots =z_\kappa = 0\}} \!\!\!\! \big(\log|z_1|^2\big)^2 \times \\
         &\qquad \qquad \qquad \qquad \ \times \frac{\partial^{2N \sum_{i=1}^\kappa m_i - 2(\kappa-1)}}{\partial z_1^{N m_1} \partial \bar{z}_1^{N m_1} \partial z_2^{N m_2 - 1} \cdots \partial \bar{z}_\kappa^{N m_\kappa - 1}} \big(\Psi \xi \big) \d z' \wedge \d \bar{z}',
    \end{align*}
    where $\d z' \wedge \d \bar{z}' = \d z_1 \wedge \d \bar{z}_1 \wedge \d z_{\kappa + 1} \wedge \d \bar{z}_{\kappa + 1} \wedge \cdots \wedge \d z_n \wedge \d \bar{z}_n$. Thus, $(I_{\kappa+1, 0})_\alpha$ is an integral over the locus $\{z_2 = \cdots = z_\kappa = 0\}$. By symmetry, it follows that $I_{\kappa+1, 0}$ is an integral over the locus
    \[
        \bigcup\limits_{i=1}^\kappa \bigcap\limits_{j\neq i} \{ z_j = 0\},
    \]
    whence
    \[
        \mathrm{supp}\,\mu_{\kappa-1}(\omega) \subseteq \bigcup\limits_{i=1}^\kappa \bigcap\limits_{j\neq i} \{ z_j = 0\} \cap \mathrm{supp}\,\Psi.
    \]
    Furthermore, since the integral $I_{\kappa,0}$ appears in both the expression for $c_{\kappa}$ and $c_{\kappa-1}$, we have that
    \[
        \mathrm{supp}\,\mu_{\kappa}(\omega) \subseteq \mathrm{supp}\,\mu_{\kappa-1}(\omega)
    \]
    %
    %
    %
    By analogous arguments for $k-r = \kappa+2,\hdots,2\kappa$ we find that 
    \[
        \mathrm{supp}\,\mu_j(\omega) \subseteq \!\!\!\!\!\!\bigcup\limits_{\substack{\ell_1,\hdots,\ell_j = 1 \\ \ell_1<\cdots<\ell_{j}}}^\kappa \!\!\!\!\!\!\{ z_{\ell_1} = \cdots = z_{\ell_j} = 0\},
    \]
    and that $\mathrm{supp}\,\mu_{j}(\omega) \subseteq \mathrm{supp}\,\mu_{j-1}(\omega)$ for each $j=1,\hdots,\kappa$.

    It is clear that $\Gamma_{\|\cdot\|}$ is holomorphic if $\mathrm{supp}\,\xi \subseteq X \setminus V$. Thus, $\mathrm{supp}\,\mu_j(\omega) \subseteq V$ for $j=1,\hdots,\kappa$. It follows that $\mu_0(\omega)$ is a current extension of $\omega$ across $V$, and we have that $\mathrm{supp}\,\mu_0(\omega) = \overline{\mathrm{supp}\,\omega}$. 
    It is shown in \cite{Lthesis} that if $V$ is smooth and $s$ defines the radical ideal of $V$ then $\kappa \leq 1$. 
    Thus, if $\mathrm{supp}\,\xi \subseteq X \setminus V_{\mathrm{sing}}$, $\Gamma_{\|\cdot\|}$ has at most a pole of order $1$ at $\lambda = 0$. This implies that $\mathrm{supp}\,\mu_j(\omega) \subseteq V_{\mathrm{sing}}$ for $j \geq 2$.
\end{proof}
Now we generalize \Cref{lem:ii} to $\omega \in \mathcal{A}_s^{p,q}(X)$. In this setting, we note that $\omega \wedge \xi \in \mathcal{A}_s^{n,n}(X)$ for any $\xi \in \mathscr{D}^{n-p,n-q}(X)$. \Cref{lem:ii} then implies that there is some $0 \leq \kappa \leq n$ and currents $\mu_j(\omega\wedge \xi)$ with compact support, for $0\leq j \leq \kappa$, which depend on $\omega$ (a priori on $\omega \wedge \xi$) such that
\[
    \Gamma_{\|\cdot\|} (\lambda) = \sum\limits_{j=0}^\kappa \frac{1}{\lambda^j} \langle \mu_j(\omega \wedge \xi),1 \rangle + \mathcal{O}(\lambda).
\]
For $\omega \in \mathcal{A}_s^{p,q}(X)$ we then define
\begin{equation}
    \label{eq:mupq}
    \langle \mu_j(\omega), \xi \rangle := \langle \mu_j(\omega \wedge \xi), 1\rangle.
\end{equation}
It is clear from the definition of $\mu_j(\omega \wedge \xi)$ that \cref{eq:mupq} defines a linear functional on $\mathscr{D}^{n-p,n-q}(X)$. Furthermore, if $\omega \in \mathcal{A}_s^{n,n}(X)$, it follows by \Cref{lem:ii} that, if $\xi$ is a test function, $\mu_j(\omega \xi) = \xi \mu_j(\omega)$, which agrees with \cref{eq:mupq}.

To see that \cref{eq:mupq} defines a $(p,q)$-current $\mu_{j}(\omega)$, it remains to check continuity. Since being a current is a local statement, we may assume that $\xi$ has support in a small neighborhood with local coordinates $z = (z_1,\hdots,z_n)$ and that 
\[
    \xi = \sum\limits_{J,K} \xi_{J K} \d z_J \wedge \d \bar{z}_K,
\]
where the sum is over all multi-indices $J$,$K$ consisting of ordered subsets of $\{1,\hdots,n\}$ of size $n-p$ and $n-q$, respectively. Since $\mu_j(\omega)$ is a linear functional, we can fix some indices $(J,K)$ and consider $\langle \mu_j(\omega), \xi_{J K} \d z_J \wedge \d \bar{z}_K \rangle$. By \cref{eq:mupq}, we have
\begin{align*}
    \langle \mu_j(\omega), \xi_{J K} \d z_J \wedge \d \bar{z}_K \rangle &= \langle \mu_j(\omega \wedge \xi_{J K} \d z_J \wedge \d \bar{z}_K), 1\rangle \\
    &= \langle \xi_{J K} \mu_j(\omega \wedge \d z_J \wedge \d \bar{z}_K), 1\rangle \\
    &= \langle \mu_j(\omega \wedge \d z_J \wedge \d \bar{z}_K), \xi_{J K} \rangle,
\end{align*}
where we used that $\mu_j(\omega \xi) = \xi \mu_j(\omega)$ for $\omega \in \mathcal{A}_s^{n,n}(X)$ and $\xi \in \mathscr{D}^{0,0}(X)$. Since we know that $\mu_j(\omega \wedge \d z_J \wedge \d \bar{z}_K)$ is a continuous linear functional on $\mathscr{D}^{0,0}(X)$, 
it follows that $\mu_j(\omega)$ is a continuous linear functional on $\mathscr{D}^{n-p,n-q}(X)$.

Thus, \Cref{lem:ii} holds for $\omega\in \mathcal{A}_s^{p,q}(X)$ with $\mu_j(\omega)$ defined as in \cref{eq:mupq}. We have the following formula.
\begin{lemma}
    \label{lem:ii2}
    Let $\omega\in\mathcal{A}_s^{p,q}(X)$. For each $j=0,\hdots,\kappa$, $\mu_j(\omega)$ satisfies
    \begin{equation}
        \label{eq:muproduct}
        \mu_j(\omega) \wedge \xi = \mu_j(\omega \wedge \xi),
    \end{equation}
    for any smooth $(p',q')$-form $\xi$.
\end{lemma}
\begin{proof}
    Let $\eta \in \mathscr{D}^{n-p-p',n-q-q'}(X)$. Since $\mu_j(\omega)$ is a $(p,q)$-current, by \cref{eq:wedgeproduct} we have that
    \[
        \langle \mu_j(\omega) \wedge \xi, \eta \rangle = \langle \mu_j(\omega), \xi \wedge \eta \rangle.
    \]
    By \cref{eq:mupq} we have
    \[
        \langle \mu_j(\omega), \xi \wedge \eta \rangle = \langle \mu_j(\omega \wedge \xi \wedge \eta), 1 \rangle.
    \]
    Since $\omega \wedge \xi \in \mathcal{A}_V^{p+p',q+q'}(X)$, again by \cref{eq:mupq}, we have that
    \[
        \langle \mu_j(\omega \wedge \xi \wedge \eta), 1 \rangle = \langle \mu_j(\omega \wedge \xi), \eta \rangle.
    \]
    %
    %
        %
    %
    Thus, $\langle \mu_j(\omega) \wedge \xi,\eta \rangle = \langle \mu_j(\omega \wedge \xi),\eta \rangle$ for all $\eta$ which proves the lemma.
\end{proof}
\begin{proof}[Proof of \Cref{thm:1} (i)]
    Let $\pi : \widetilde{X} \rightarrow X$ be a modification such that $\widetilde{X}$ is smooth and $\pi^*s : \widetilde{X} \rightarrow \pi^* E$ defines a normal crossings divisor. As in the proof of \Cref{prop:A}, with $\tau = 0$, we have, for $\mathfrak{Re}\,\lambda \gg 0$,
    \[
        \Gamma_{\|\cdot\|}(\lambda) = \int\limits_X \|s\|^{2\lambda} \omega \wedge \xi = \int\limits_{\widetilde{X}} \|\pi^*s\|^{2\lambda} \pi^* \omega \wedge \pi^*\xi.
    \]
    Since $\pi^*\omega \wedge \pi^*\xi \in \mathcal{A}_{\pi^{*}s}^{n,n}(\widetilde{X})$, by \Cref{lem:ii} $\Gamma_{\|\cdot\|}(\lambda)$ has a meromorphic continuation to $\C_\lambda$, with polar set given by a discrete subset $P \subset \Q \cap (-\infty,N]$ for some $N \geq 0$. Moreover, there is some $0 \leq \kappa\leq n$ such that, in a neighborhood of $\lambda = 0$,
    \[
        \Gamma_{\|\cdot\|}(\lambda) = \sum\limits_{j = 1}^{\kappa} \frac{1}{\lambda^j}\langle \mu_j(\pi^*\omega \wedge \pi^* \xi),1 \rangle + \mathcal{O}(\lambda),
    \]
    where $\mu_j(\pi^*\omega \wedge \pi^* \xi)$ define $(n,n)$-currents on $\widetilde{X}$. By \cref{eq:mupq} we may write
    \[
        \Gamma_{\|\cdot\|}(\lambda) = \sum\limits_{j = 1}^{\kappa} \frac{1}{\lambda^j} \langle \mu_j(\pi^*\omega),\pi^* \xi \rangle + \mathcal{O}(\lambda),
    \]
    where $\mu_j(\pi^*\omega)$ are $(p,q)$-currents on $\widetilde{X}$. Since $\pi$ is proper, by \cref{eq:pushforward} we have that
    \[
        \langle \mu_j(\pi^*\omega),\pi^* \xi \rangle = \langle \mu_j(\omega),\xi \rangle,
    \]
    where 
    \begin{equation}
        \label{eq:mupushforward}
        \mu_j(\omega) := \pi_* \mu_j(\pi^*\omega)
    \end{equation}
    is a current on $X$, for each $j=0,\hdots,\kappa$.

    By \Cref{lem:ii} we have that $\mathrm{supp}\,\mu_0(\pi^*\omega) = \overline{\mathrm{supp}\,\pi^* \omega}$, $\mathrm{supp}\,\mu_1(\pi^*\omega) \subseteq \pi^{-1}V$, and $\mathrm{supp}\,\mu_\kappa(\pi^*\omega) \subseteq \cdots \subseteq \mathrm{supp}\,\mu_2(\pi^*\omega) \subseteq (\pi^{-1}V)_{\mathrm{sing}}$.  Furthermore, it follows immediately by taking direct images that $\mathrm{supp}\,\mu_0(\omega) = \mathrm{supp}\,\pi_*\mu_0(\pi^*\omega) = \overline{\mathrm{supp}\,\omega}$, and $\mathrm{supp}\,\mu_\kappa(\omega) \subseteq \cdots \subseteq \mathrm{supp}\,\mu_0(\omega)$.
    
    It is shown in \cite{Lthesis} that if $X$ is smooth, and $V$ is a submanifold, then $\Gamma_{\|\cdot\|}(\lambda)$ has a pole of order at most $1$ at the origin. Thus, it follows that $\mathrm{supp}\,\mu_1(\omega) \subset V$ and $\mathrm{supp}\,\mu_j(\omega) \subseteq V_{\mathrm{sing}} \cup (X_{\mathrm{sing}}\cap V)$ for each $j\geq 2$.
    
\end{proof}
A priori \Cref{lem:ii2} holds in the case when $X$ is smooth and $s$ defines a normal crossings divisor. The corresponding statement in the general setting follows by \Cref{lem:ii2} and \cref{eq:pushforwardofproduct}.
%
    %
    %
    %
    %
%

\subsection{Proof of \Cref{thm:1} (ii)}

\begin{lemma}
    \label{lem:iii}
    For $\omega \in \mathcal{A}_{s,\|\cdot\|}(X)$ we have that $\d \omega$, $\frac{\d \|s\|^2}{\|s\|^2} \wedge \omega \in \mathcal{A}_{s,\|\cdot\|}(X)$.
\end{lemma}
\begin{proof}
    Since $\omega \in \mathcal{A}_{s,\|\cdot\|}(X)$, for each compact $K \subset X$ we can find an integer $N \geq0$ such that $\|s\|^{2N} \omega$ extends smoothly across $V\cap K$. Thus, we may write
    \[
        \omega = \frac{\widetilde{\omega}}{\|s\|^{2N}},
    \]
    where $\widetilde{\omega}$ is smooth across $K \cap V$. On $X \setminus V$ we have that
    \begin{align*}
        \d \omega &= \d \frac{\widetilde{\omega}}{\|s\|^{2N}} = \frac{\d \widetilde{\omega}}{\|s\|^{2N}} - N \frac{\d \|s\|^2}{\|s\|^2} \wedge \frac{\widetilde{\omega}}{\|s\|^{2N}}.
    \end{align*}
    Since $\d \widetilde{\omega}$ and $\d \|s\|^2$ are smooth across $V \cap K$, it is clear that
    \[
        \|s\|^{2(N+1)}\d \omega = \|s\|^2 \d \widetilde{\omega} - N \d\|s\|^2 \wedge \widetilde{\omega}
    \]
    extends smoothly across $V \cap K$. It is also clear that $\|s\|^{2(N+1)}\frac{\d\|s\|^2}{\|s\|^2} \wedge \omega$ extends smoothly across $V\cap K$.
\end{proof}
\begin{remark}
    For $\omega \in \mathcal{A}_{s,\|\cdot\|}(X)$ and $|\cdot|$ some different metric on $E$, it is not true in general that $\frac{\d |s|^2}{|s|^2}\wedge \omega \in \mathcal{A}_s(X)$. However, we can always find an integer $N\geq 0$ such that $|s|^{2N}\frac{\d |s|^2}{|s|^2}\wedge \omega$ extends to a locally bounded form on $X$, and for a modification $\pi : \widetilde{X} \rightarrow X$ such that $\pi^* s$ defines a divisor, the pullback of $|s|^{2N}\frac{\d |s|^2}{|s|^2}\wedge \omega$ is smooth for large $N$, that is, $\pi^* \Big( \frac{\d |s|^2}{|s|^2}\wedge \omega \Big)\in \mathcal{A}_{\pi^{*}s}(\widetilde{X})$. 
\end{remark}
\begin{proof}[Proof of \Cref{thm:1} (ii)]
    Let $\xi \in \mathscr{D}^{n-p,n-q}(X)$. Then $\exists N\geq 0$ such that $\|s\|^{2N}\omega$ extends smoothly across $V \cap \mathrm{supp}\,\xi$. Using integration by parts and Stokes' theorem, we have, for $\mathfrak{Re}\, \lambda \gg 0$,
    \begin{align*}
        \int\limits_{X} \|s\|^{2\lambda} \omega \wedge \d \xi &= (-1)^{p+q+1} \int\limits_{X} \d (\|s\|^{2\lambda}\omega) \wedge \xi \\
        &= (-1)^{p+q+1} \lambda \int\limits_{X} \|s\|^{2\lambda} \frac{\d \|s\|^2}{\|s\|^2} \wedge \omega \wedge \xi + (-1)^{p+q+1} \int\limits_{X} \|s\|^{2\lambda} \d \omega \wedge \xi.
    \end{align*}
    By \Cref{lem:iii}, $\d \omega, \frac{\d \|s\|^2}{\|s\|^2} \wedge \omega \in \mathcal{A}_s(X)$. Thus, by \Cref{thm:1} (i), and by uniqueness of meromorphic continuation, we obtain the following equality of Laurent series expansions about $0$,
    \[
        \sum\limits_{j=0}^{\kappa}\frac{1}{\lambda^j} \langle \d \mu_j(\omega),\xi \rangle = \sum\limits_{j=1}^{\kappa'} \frac{1}{\lambda^{j-1}} \langle \mu_j\left( \frac{\d \|s\|^2}{\|s\|^2} \wedge \omega \right), \xi \rangle + \sum\limits_{j=0}^{\kappa''} \frac{1}{\lambda^j} \langle \mu_j(\d \omega),\xi \rangle + \mathcal{O}(\lambda),
    \]
    %
    where we have used \cref{eq:dofcurrent} on the left-hand side. Collecting the terms by order in $\lambda$, we obtain the equality \cref{eq:dmu} for each $j$.
\end{proof}

\section{Proof of \Cref{thm:2}}

In this section we give the proof of our main result, \Cref{thm:2}.
\begin{proof}[Proof of \Cref{thm:2}]
    Recall that $\omega \in \mathcal{A}_s(X)$ and that $\|\cdot\|$ and $|\cdot|$ are two smooth Hermitian metrics on $E$. Let $\xi$ be a test form of complementary bidegree to $\omega$ and consider
    \[
        \Gamma(\lambda,\tau) = \int\limits_X \|s \|^{2\lambda} \left( \frac{|s|}{\|s\|}\right)^{\!\!2\tau}\!\! \omega \wedge \xi.
    \]
    By \Cref{prop:A}, $\Gamma(\lambda,\tau)$ is holomorphic if $\mathfrak{Re}\,\lambda \gg 0$ and extends to a meromorphic function on $\C^2$. Furthermore, there is a discrete subset $P \subset \Q \cap (-\infty, N]$, for some $N\geq 0$ such that the polar locus of $\Gamma(\lambda,\tau)$ lies in $P \times \C_\tau$.
    \medskip
    
    Suppose first that $X$ is smooth and that $s$ defines a normal crossings divisor. 
    Then $|s|^2/\|s\|^2$ is a smooth positive function on $X$. By \Cref{thm:1} (i), for each fixed $\tau \in \C$ and $\mathfrak{Re}\,\lambda \gg 0$, there is some $\kappa' \leq n$ such that
    \begin{equation}
        \label{eq:lambdatau1}
        \int\limits_{X} \|s\|^{2\lambda} \left( \frac{|s|^2}{\|s\|^2}\right)^{\!\!\tau} \omega \wedge \xi = \sum\limits_{j=0}^{\kappa'} \frac{1}{\lambda^j} \langle \mu_j^{\|\cdot\|}\bigg( \!\!\left( \frac{|s|^2}{\|s\|^2}\right)^{\!\!\tau}\omega \bigg),\xi \rangle + F(\lambda,\tau),
    \end{equation}
    where $\lambda \mapsto F(\lambda,\tau)$ is meromorphic in $\C_\lambda$, holomorphic  for $\lambda$ near $0$ and $F(0,\tau) = 0$. By \Cref{lem:ii2} we have that
    \begin{equation}
        \label{eq:lambdatau2}
        \sum\limits_{j=0}^{\kappa'} \frac{1}{\lambda^j} \langle \mu_j^{\|\cdot\|}\bigg( \!\!\left( \frac{|s|^2}{\|s\|^2}\right)^{\!\!\tau}\omega \bigg),\xi \rangle + F(\lambda,\tau) = \sum\limits_{j=0}^{\kappa'} \frac{1}{\lambda^j} \langle \left( \frac{|s|^2}{\|s\|^2}\right)^{\!\!\tau} \mu_j^{\|\cdot\|}(\omega),\xi \rangle + F(\lambda,\tau).
    \end{equation}
    The left hand side of \cref{eq:lambdatau1} is meromorphic by \Cref{prop:A} with polar set $P \times \C_\tau$. Each term in the sum in the right hand side of \cref{eq:lambdatau2} is meromorphic in $\C^2$ with polar set $\{0\} \times \C_\tau$. It follows that $F(\lambda,\tau)$ is meromorphic in $\C^2$ with polar set $(P \setminus \{0\}) \times \C_\tau$. On the line $\tau = \lambda$ in $\C^2$ we obtain an equality of meromorphic functions
    \[
        \int\limits_{X} |s|^{2\lambda} \omega \wedge \xi = \sum\limits_{j=0}^{\kappa'} \frac{1}{\lambda^j} \langle \left( \frac{|s|^2}{\|s\|^2}\right)^{\!\!\lambda} \mu_j^{\|\cdot\|}(\omega),\xi \rangle + F(\lambda,\lambda),
    \]
    where $F(0,0) = 0$, and $F(\lambda,\lambda)$ is holomorphic for $\lambda$ near 0. Thus, the sum on the right hand side contains the principal part of the Laurent series expansion of the left hand side around $\lambda = 0$. But, by \Cref{thm:1} (i), the Laurent series expansion of the left hand side is given by
    \[
        \sum\limits_{j=0}^\kappa \frac{1}{\lambda^j} \langle \mu_j^{|\cdot|}(\omega),\xi \rangle + \mathcal{O}(\lambda),
    \]
    for some $\kappa \leq n$. Thus, since
    \[
        \left( \frac{|s|^2}{\|s\|^2}\right)^{\!\!\lambda} = \sum\limits_{\ell=0}^\infty \frac{\lambda^\ell}{\ell!} \left( \log \frac{|s|^2}{\|s\|^2}\right)^{\!\!\ell},
    \]
    by uniqueness of Laurent series expansions we have that
    \begin{equation}
        \label{eq:comparison}
        \mu_j^{|\cdot|}(\omega) = \sum\limits_{\ell=0}^{\kappa'-j} \frac{1}{\ell!} \left( \log \frac{|s|^2}{\|s\|^2}\right)^{\!\!\ell} \mu_{j+\ell}^{\|\cdot\|}(\omega).
    \end{equation}
    It immediately follows that $\kappa' = \kappa$, that is, $\kappa$ is independent of the metric when $s$ defines a divisor, and,
    as a consequence $\mu_\kappa(\omega) := \mu_\kappa^{|\cdot|}(\omega)$ is independent of the choice of metric.
    
    \medskip
    Now, for the general case: Let $\pi:\widetilde{X} \rightarrow X$ be a modification such that $\widetilde{X}$ is smooth and $\pi^* s$ defines a normal crossings divisor. Then \cref{eq:comparison} holds with $\omega$ and $s$ replaced by $\pi^* \omega$ and $\pi^* s$, respectively. In view of \cref{eq:mupushforward}, we have that $\mu_j^{|\cdot|}(\omega) = \pi_* \mu_j^{|\cdot|}(\pi^*\omega)$ for $j=0,\hdots,\kappa$ and $\mu_j^{\|\cdot\|}(\omega) = \pi_* \mu_j^{\|\cdot\|}(\pi^*\omega)$ for $j=0,\hdots,\kappa'$. It follows that, for each $j=0,\hdots,\kappa$,
    \begin{equation}
        \label{eq:n-j}
        \begin{aligned}
            \mu_j^{|\cdot|}(\omega) &= \pi_* \sum\limits_{\ell=0}^{\kappa'-j} \frac{1}{\ell!}\left( \log \frac{|\pi^*s|^2}{\|\pi^*s\|^2}\right)^{\!\!\ell} \mu_{j+\ell}^{\|\cdot\|}(\pi^*\omega) = \sum\limits_{\ell=0}^{\kappa'-j} \frac{1}{\ell!} \left( \log\frac{|s|^2}{\|s\|^2}\right)^{\!\!\ell} \mu_{j+\ell}^{\|\cdot\|}(\omega),
        \end{aligned}
    \end{equation}
    where 
    \[
        \left( \log\frac{|s|^2}{\|s\|^2}\right)^{\!\!\ell} \mu_{j+\ell}^{\|\cdot\|}(\omega) := \pi_* \bigg( \!\!\left( \log \frac{|\pi^*s|^2}{\|\pi^*s\|^2}\right)^{\!\!\ell} \mu_{j+\ell}^{\|\cdot\|}(\pi^*\omega) \bigg)
    \]
    according to \cref{eq:generalizedproduct}. 
\end{proof}
Note that if $\kappa > \kappa'$, even though $\mu^{\|\cdot\|}_{j+\ell}(\omega) = 0$ for $\ell > \kappa' - j$, this does not immediately imply that
\[
    \left( \log\frac{|s|^2}{\|s\|^2}\right)^{\!\!\ell} \mu_{j+\ell}^{\|\cdot\|}(\omega) = 0,
\]
for $\kappa' - j < \ell \leq \kappa - j$, cf. \Cref{ex:2}. Thus, it is not clear in general whether $\kappa (\leq n)$ is independent of the choice of metric, unless $V = \{s=0\}$ is a hypersurface, in which case $\log\frac{|s|^2}{\|s\|^2}$ is smooth.
%
%

\medskip
The dependence of the currents $\mu_j(\omega)$ on the choice of section defining $V$ is, in fact, essentially described by \Cref{thm:2}, in a sense which we try to illustrate with the following example.
\begin{example}
    \label{ex:3}
    Suppose that $V$ is a hypersurface and that there are (holomorphic) line bundles $E$ and $F$ and (holomorphic) sections $s : X \rightarrow E$ and $\sigma : X \rightarrow F$ such that $V = \{s=0\} = \{\sigma=0\}$ and such that $\sigma$ and $s^{\otimes k}$ define the same divisor, for some $k \in \mathbb{N}$. Moreover, let $|\cdot|_E$ and $|\cdot|_F$ be Hermitian metrics on $E$ and $F$, respectively, and suppose that $\omega \in \mathcal{A}_s(X) = \mathcal{A}_\sigma(X)$. 

    The metric $|\cdot|_E$ naturally induces a metric $|\cdot|_{E^{\otimes k}}$ on $E^{\otimes k}$ satisfying $|s^{\otimes k}|^2_{E^{\otimes k}} = |s|_E^{2k}$. Thus, since $\sigma$ and $s^{\otimes k}$ define the same divisor, we have that
    \[
        \frac{|\sigma|^2_F}{|s^{\otimes k}|_{E^{\otimes k}}} = \frac{|\sigma|^{2}_F}{|s|^{2k}_E}
    \]
    is a smooth positive function on $X$.
    Thus, we can define a new metric $\|\cdot\|_E$ on $E$ by
    \[
        \|v\|_E^2 := |v|^2_E \frac{|\sigma|^{2/k}_F}{|s|^{2}_E},
    \]
    for $v \in H^0(X,E)$. For $\mathfrak{Re}\,\lambda \gg 0$, we then have that
    \begin{align*}
        \int\limits_X |\sigma|^{2\lambda}_F \,\omega \wedge \xi &= \int\limits_X |s^{\otimes k}|^{2\lambda}_{E^{\otimes k}} \bigg( \frac{|\sigma|^2_F}{|s^{\otimes k}|^2_{E^{\otimes k}}}\bigg)^{\!\!\lambda} \omega \wedge \xi \\
        &= \int\limits_X |s|^{2k\lambda}_{E} \bigg( \frac{|\sigma|^2_F}{|s|^{2k}_{E}}\bigg)^{\!\!\lambda} \omega \wedge \xi = \int\limits_X \|s\|_E^{2k\lambda} \omega\wedge \xi.
    \end{align*}
    Thus, we see that the change of sections, from $s$ to $\sigma$, can be realized as a change in metrics on $E$, keeping the section $s$ fixed, after a possible rescaling of $\lambda$.
    
    %
\end{example}

\section{Asymptotic expansion of $\mathcal{I}_{\|\cdot\|}(\epsilon)$}

Recall that $\mathcal{I}_{\|\cdot\|}(\epsilon)$ is given by \cref{eq:cutoff}, where $\omega \in \mathcal{A}_{s}(X)$, $\xi \in \mathscr{D}(X)$, $s$ is a holomorphic section of $E$ such that $\{s=0\} = V$ and $\|\cdot\|$ is a smooth Hermitian metric on $E$. In this section we relate the asymptotic expansion of $\mathcal{I}_{\|\cdot\|}(\epsilon)$ to the Laurent series expansion \cref{eq:gammalaurent} of $\Gamma_{\|\cdot\|}(\lambda)$.
\begin{proposition}
    \label{thm:3}
    Let $P \subset \Q$ denote the polar set of $\Gamma_{\|\cdot\|}(\lambda)$ and let $P_+ = P \cap \{ \mathfrak{Re}\,\lambda > 0\}$. We have that
    \begin{equation}
        \label{eq:asym}
        \mathcal{I}_{\|\cdot\|}(\epsilon) = \langle \mu_0^{\|\cdot\|}(\omega),\xi\rangle + \sum\limits_{j=1}^\kappa \frac{1}{j!} \big( \!\log\epsilon^{-1}\big)^{j} \langle \mu_j^{\|\cdot\|}(\omega),\xi\rangle + \!\sum\limits_{p \in P_+} \underset{\lambda = p}{\mathrm{Res}} \Big\{ \epsilon^{-\lambda} \lambda^{-1} \Gamma_{\|\cdot\|}(\lambda)\Big\} + \mathcal{O}(\epsilon^{\delta}),
    \end{equation}
    for some $\delta > 0$.
\end{proposition}
%
\medskip
As we show below,
\[
    \underset{\lambda = p}{\mathrm{Res}}\Big\{ \epsilon^{-\lambda} \lambda^{-1} \Gamma_{\|\cdot\|}(\lambda) \Big\} = \epsilon^{-p} \sum_{j=0}^{2\ell_p - 1} \frac{1}{j!} \big( \log \epsilon^{-1}\big)^j c_{2\ell_p - 1 -j}
\]
where $\ell_p \in \mathbb{N}$ and where the $c_{2\ell_p - 1 -j}$ are independent of $\epsilon$. If $V$ is a hypersurface, the existence of an asymptotic expansion of $\mathcal{I}_{\|\cdot\|}(\epsilon)$ of this form follows from \cite[Theorem 4.3.1]{BA1}. The proof of that theorem is based on \cite{BA2} and the existence of Bernstein--Sato polynomials. It is reasonable to expect that \Cref{thm:3} can be proven in a similar way. We have instead chosen to use the fact that $\mathcal{I}_{\|\cdot\|}(\epsilon)$ and $\Gamma_{\|\cdot\|}(\lambda)$ are related via the Mellin transform.

The first observation is that $\Gamma_{\|\cdot\|}(\lambda)$ satisfies a certain growth condition.
\begin{lemma}
    \label{lem:iv}
    The function $\Gamma_{\|\cdot\|}(\lambda)$ is rapidly decreasing in $\mathfrak{Im}\,\lambda$, in the sense that the product $\lambda^\ell \Gamma_{\|\cdot\|}(\lambda)$, for any $\ell \in \mathbb{N}$, is a bounded function when $\lambda = \alpha + i \beta$ and $|\beta| \rightarrow \infty$, locally uniformly in $\alpha$.
\end{lemma}
The following proof is an adaptation of the proof of Lemma 6.1 in \cite{A1}. 
%
\begin{proof}
    Let $\pi : \widetilde{X} \rightarrow X$ be a modification such that $\widetilde{X}$ is smooth and $\pi^* s$ defines a normal crossings divisor. Recall that then
    \[
        \Gamma_{\|\cdot\|}(\lambda) = \int\limits_{\widetilde{X}} \|\pi^*s\|^{2\lambda} \pi^* \omega \wedge \pi^* \xi.
    \]
    Locally in $\widetilde{X}$ we can choose coordinates such that 
    $\|\pi^*s\|^{2\lambda} = |z^{m_1}_1,\hdots,z^{m_\kappa}_\kappa|^{2\lambda} e^{-\lambda\phi}$, for some $1\leq \kappa \leq n$, $m_1,\hdots,m_\kappa \geq 1$, and $\phi$ a local weight associated to $\|\cdot\|$, and
    \[
        \pi^*\omega = \frac{\Psi \,\d z\wedge \d \bar{z}}{|z^{m_1}_1,\hdots,z^{m_\kappa}_\kappa|^{2N}},
    \]
    for some smooth function $\Psi$ and some integer $N \geq 0$. By introducing a partition of unity $\{\rho_j\}$ on $\widetilde{X}$, $\Gamma_{\|\cdot\|}(\lambda)$ can be written as a finite sum of terms of the form
    \[
        \int\limits_{\C^n} |z^{m_1}_1 \cdots z^{m_{\kappa}}_\kappa|^{2(\lambda-N)} e^{-\lambda \phi} \Psi \pi^* \xi \rho_j \, \d z \wedge \d \bar{z}.
    \]
    Notice that $\kappa$, $m_1,\hdots,m_\kappa$, $N$, $\phi$ and $\Psi$ all depend on $j$.
    
    Consider the (non-holmorphic) change of variables, $\sigma_1 = e^{-\phi/2m_1} z_1$, $\sigma_\ell = z_\ell$ for $2 \leq \ell \leq n$. We have that $\d \sigma_\ell = \d z_\ell$ for $2\leq \ell \leq n$, and
    \[
        \d \sigma_1 = e^{-\phi/2m_1} \d z_1 - \frac{1}{2m_1}e^{-\phi/2m_1} z_1 \sum\limits_{\ell=1}^n  \bigg( \frac{\partial \phi}{\partial z_\ell} \d z_\ell + \frac{\partial \phi}{\partial \bar{z}_\ell} \d \bar{z}_\ell\bigg).
    \]
    It follows that
    \[
        \d \sigma \wedge \d \bar{\sigma} = e^{-\phi/m_1} \bigg( 1 - \frac{1}{m_1}\mathfrak{Re}\,z_1 \frac{\partial \phi}{\partial z_1}\bigg) \d z\wedge \d \bar{z}.
    \]
    We can take $\rho_j$ to be such that $\d \sigma \wedge \d \bar{\sigma} \neq 0$ on $\mathrm{supp}\,\rho_j$. We then have that
    \[
        \int\limits_{\C^n} |z^{m_1}_1 \cdots z^{m_{\kappa}}_\kappa|^{2(\lambda-N)} e^{-\lambda \phi} \Psi \pi^* \xi \rho_j \, \d z \wedge \d \bar{z} = \int\limits_{\C^n} |\sigma^{m_1}_1 \cdots \sigma^{m_\kappa}_\kappa|^{2(\lambda-N)} \widetilde{\Psi} \pi^* \xi \rho_j \, \d \sigma \wedge \d \bar{\sigma},
    \]
    where
    \[
        \widetilde{\Psi} = \bigg( 1 - \frac{1}{m_1}\mathfrak{Re}\,z_1 \frac{\partial \phi}{\partial z_1}\bigg)^{\!\!-1}\!\! e^{-(N-1/m_1)\phi}\Psi,
    \]
    is smooth on $\mathrm{supp}\,\rho_j$. Following the steps in the proof of \Cref{lem:i}, for any positive integer $M$ we have by \cref{eq:merocontM1} that
    \begin{align*}
        \int\limits_{\C^n} |\sigma^{m_1}_1 \cdots \sigma^{m_\kappa}_\kappa|^{2(\lambda-N)} \widetilde{\Psi} \pi^* \xi \rho_j \, \d \sigma \wedge \d \bar{\sigma} 
        &= \frac{h_M(\lambda)}{\lambda^{2\kappa}}\int\limits_{\C^n} |\sigma^{m_1}_1 \cdots \sigma^{m_\kappa}_\kappa|^{2(\lambda + M)} \times \\
        &\qquad \times P_M \bar{P}_M \big( \widetilde{\Psi} \pi^* \xi \rho_j \big) \d \sigma \wedge \d \bar{\sigma},
    \end{align*}
    where
    \[
        P_M \bar{P}_M = \frac{\partial^{2(N+M)\sum_{i=1}^\kappa m_i}}{\partial \sigma_1^{(N+M)m_1} \cdots \partial \bar{\sigma}_\kappa^{(N+M)m_\kappa} },
    \]
    and where $h_M(\lambda)$ is given by \cref{eq:hMlambda}. Notice that
    \[
        \frac{|h_M(\lambda)|}{|\lambda|^{2\kappa}} = \mathcal{O}\Big(|\lambda|^{-2(N+M)\sum_{i=1}^\kappa m_i}\Big),
    \]
    for $|\lambda|\gg 0$. For $\lambda = \alpha + i \beta$ with $\alpha > -M $, the integral
    \[
        \int\limits_{\C^n} |\sigma^{m_1}_1 \cdots \sigma^{m_\kappa}_\kappa|^{2(\lambda + M)} P_M \bar{P}_M \big( \widetilde{\Psi} \pi^* \xi \rho_j \big) \d \sigma \wedge \d \bar{\sigma}
    \]
    is finite, and it clearly remains finite if we let $|\beta| \rightarrow \infty$, locally uniformly in $\alpha$. Thus, for $|\beta| \gg 0$, $|\Gamma_{\|\cdot\|}(\alpha+i\beta)| = \mathcal{O}(|\beta|^{-2(N+M)\sum_{i=1}^\kappa m_i})$. As $M$ was chosen arbitrarily, it follows that $(\alpha+i\beta)^\ell \Gamma_{\|\cdot\|}(\alpha + i\beta)$ is a bounded function when $|\beta| \rightarrow \infty$ for any $\ell \in \mathbb{N}$.
\end{proof}
%

As mentioned, the functions $\Gamma_{\|\cdot\|}(\lambda)$ and $\mathcal{I}_{\|\cdot\|}(\epsilon)$ are related via the Mellin transform. The Mellin transform of a function $f$ defined on $\mathbb{R}_+$ is given by
\[
    \{\mathcal{M}f\}(\lambda) = \int\limits_0^\infty \epsilon^{\lambda - 1} f(\epsilon) \d \epsilon.
\]
Notice that if $|f(\epsilon)| \lesssim \epsilon^{-N}$ for some $N \geq 0$ as $\epsilon \rightarrow 0$ and $f(\epsilon)=0$ if $\epsilon \gg 0$, then $\{\mathcal{M}f \}(\lambda)$ is holomorphic for $\mathfrak{Re}\,\lambda \gg 0$.

If a function $\varphi : \C \rightarrow \C$ is holomorphic in the strip $a < \mathfrak{Re}\,\lambda < b$, and if it tends to zero uniformly as $|\mathfrak{Im}\,\lambda| \rightarrow \infty$, for $\mathfrak{Re}\,\lambda = c$, where $c \in (a,b)$, such that its integral along such a line is absolutely convergent, then $\varphi$ has an inverse Mellin transform, given by
\[
    \{\mathcal{M}^{-1}\varphi \}(\epsilon) = \frac{1}{2\pi i} \int\limits_{c- i \infty}^{c+ i \infty} \epsilon^{-\lambda} \varphi(\lambda) \d \lambda.
\]
\begin{lemma}
    \label{lem:v}
    We can find an integer $N \geq 0$ such that $|\mathcal{I}_{\|\cdot\|}(\epsilon)| \lesssim \epsilon^{-2N}$ as $\epsilon \rightarrow 0$; additionally we have that $\mathcal{I}_{\|\cdot\|}(\epsilon) = 0$ for $\epsilon \gg 0$. 
    For $\mathfrak{Re}\,\lambda \gg 0$, we have that
    \[
        \{\mathcal{M} \mathcal{I}_{\|\cdot\|}\}(\lambda) = \frac{1}{\lambda}\Gamma_{\|\cdot\|}(\lambda).
    \]
\end{lemma}
%
This relation between the two regularization methods considered is well known. It frequently appears in the context of residue theory, see \cite{A1,PT,Yger}, but it has also been recognized in the context of divergent integrals, see, e.g., \cite{BA1, FK2}.
\begin{proof}
    Since $\omega \in \mathcal{A}_s(X)$ we can find an integer $N \geq 0$ such that $\omega = \widetilde{\omega}/\|s\|^{2N}$ where $\widetilde{\omega}$ is bounded on $\mathrm{supp}\,\xi$. Thus,
    \[
        |\mathcal{I}_{\|\cdot\|}(\epsilon)| \leq \int\limits_{\|s\|^2 \geq \epsilon} \frac{|\widetilde{\omega} \wedge \xi|}{\|s\|^{2N}} \lesssim \frac{1}{\epsilon^{2N}}.
    \]
    Since $\widetilde{\omega} \wedge \xi$ has compact support, $\mathcal{I}_{\|\cdot\|}(\epsilon) = 0$ for $\epsilon \gg 0$. 
    %
    %
    
    By Fubini's theorem
    \begin{align*}
        \{ \mathcal{M}\mathcal{I}_{\|\cdot\|}\}(\lambda) &= \int\limits_0^\infty \epsilon^{\lambda-1} \int \limits_{\|s\|^2 \geq \epsilon} \omega \wedge \xi \,\d \epsilon \\
        &= \int_X \int\limits_{0}^{\|s\|^2} \epsilon^{\lambda-1}\d \epsilon \, \omega \wedge \xi = \frac{1}{\lambda} \int\limits_X \|s\|^{2\lambda} \omega \wedge \xi = \frac{1}{\lambda} \Gamma_{\|\cdot\|}(\lambda),
    \end{align*}
    for $\mathfrak{Re}\,\lambda \gg 0$.
\end{proof}
By \Cref{lem:iv,lem:v} it follows that $\mathcal{I}_{\|\cdot\|}(\epsilon)$ can be recovered from $\Gamma_{\|\cdot\|}(\lambda)$ via the inverse Mellin transform as follows,
\begin{equation}
    \label{eq:invmellin}
    \mathcal{I}_{\|\cdot\|}(\epsilon) = \big\{ \mathcal{M}^{-1} \lambda^{-1} \Gamma_{\|\cdot\|}(\lambda)\big\}(\epsilon) = \frac{1}{2\pi i} \int_{c - i \infty}^{c+ i \infty} \epsilon^{-\lambda} \lambda^{-1} \Gamma_{\|\cdot\|}(\lambda) \d \lambda,
\end{equation}
for $c \gg 0$. With \cref{eq:invmellin}, we are ready to prove \Cref{thm:3}.
\begin{proof}[Proof of \Cref{thm:3}]
    It follows from \Cref{thm:1} (i) that $\epsilon^{-\lambda}\lambda^{-1} \Gamma_{\|\cdot\|}(\lambda)$ defines a meromorphic function with polar set $P$ contained in $\Q \cap (-\infty,N]$ for some $N \geq 0$. Let $\delta > 0$ such that $\Gamma_{\|\cdot\|}(\lambda)$ has no poles in the interval $[-\delta,0)$ and let $c \gg 1$ such that \cref{eq:invmellin} holds. Let $B = \{ -\delta < \mathfrak{Re}\, \lambda < c\} \subset \C$ and let $\partial B$ be the positively oriented boundary of $B$. By the Residue theorem and \Cref{lem:iv} we have that
    \begin{equation}
        \label{eq:residues}
        \frac{1}{2\pi i} \oint\limits_{\partial B} \epsilon^{-\lambda} \lambda^{-1} \Gamma_{\|\cdot\|}(\lambda) \d \lambda = \sum\limits_{p \in P \cap B} \underset{\lambda = p}{\mathrm{Res}} \Big\{ \epsilon^{-\lambda} \lambda^{-1} \Gamma_{\|\cdot\|}(\lambda) \Big\}.
    \end{equation}
    By a straightforward computation
    \begin{align*}
         \frac{1}{2\pi i} \oint\limits_{\partial B} \epsilon^{-\lambda} \lambda^{-1} \Gamma_{\|\cdot\|}(\lambda)\d \lambda &= \frac{1}{2\pi i} \int\limits_{\mathclap{c-i \infty}}^{\mathclap{c+i \infty}}\epsilon^{-\lambda} \lambda^{-1} \Gamma_{\|\cdot\|}(\lambda) \d \lambda + \mathcal{O}(\epsilon^\delta).
    \end{align*}
    Thus, by \cref{eq:invmellin,eq:residues}, it follows that
    \begin{equation}
        \label{eq:Iepsilon}
        \mathcal{I}_{\|\cdot\|}(\epsilon) = \sum\limits_{p \in P \cap B} \underset{\lambda = p}{\mathrm{Res}} \Big\{ \epsilon^{-\lambda} \lambda^{-1} \Gamma_{\|\cdot\|}(\lambda) \Big\} + \mathcal{O}(\epsilon^{\delta}).
    \end{equation}
    Let $P_+ = P \cap \{ \mathfrak{Re}\,\lambda > 0 \}$; we write
    \[
        \sum\limits_{p \in P \cap B} \underset{\lambda = p}{\mathrm{Res}} \Big\{ \epsilon^{-\lambda} \lambda^{-1} \Gamma_{\|\cdot\|}(\lambda) \Big\} = \underset{\lambda = 0}{\mathrm{Res}} \Big\{ \epsilon^{-\lambda} \lambda^{-1} \Gamma_{\|\cdot\|}(\lambda) \Big\} + \sum\limits_{p \in P_+} \underset{\lambda = p}{\mathrm{Res}} \Big\{ \epsilon^{-\lambda} \lambda^{-1} \Gamma_{\|\cdot\|}(\lambda) \Big\},
    \]
    where, by \Cref{thm:1} (i), we have that
    \begin{align*}
        \underset{\lambda = 0}{\mathrm{Res}} \Big\{ \epsilon^{-\lambda} \lambda^{-1} \Gamma_{\|\cdot\|}(\lambda) \Big\} &= \underset{\lambda = 0}{\mathrm{Res}} \Bigg\{ \sum\limits_{\ell=0}^\infty \frac{1}{\ell!} \lambda^{\ell-1} \big( \log \epsilon^{-1} \big)^\ell \bigg( \sum\limits_{j=0}^\kappa \lambda^{-j} \langle \mu_j^{\|\cdot\|}(\omega),\xi \rangle + \mathcal{O}(\lambda) \bigg) \Bigg\} \\
        &= \underset{\lambda = 0}{\mathrm{Res}} \Bigg\{ \sum\limits_{j=0}^\kappa  \langle \mu_j^{\|\cdot\|}(\omega),\xi \rangle \sum\limits_{\ell=0}^j \frac{1}{\ell!} \lambda^{\ell-j-1} \big( \log \epsilon^{-1} \big)^\ell + \mathcal{O}(1) \Bigg\} \\
        &= \sum\limits_{j=0}^\kappa  \frac{1}{j!} \big( \log \epsilon^{-1} \big)^j \langle \mu_j^{\|\cdot\|}(\omega),\xi \rangle.
    \end{align*}
    \Cref{thm:3} now follows in view of \cref{eq:Iepsilon}.
\end{proof} 
We can look more closely at the residues 
\[
    \underset{\lambda = p}{\mathrm{Res}} \Big\{ \epsilon^{-\lambda} \lambda^{-1} \Gamma_{\|\cdot\|}(\lambda) \Big\},
\]
for $p\in P_+$. Following the proof of \Cref{lem:ii} and \Cref{thm:1}, let $\pi : \widetilde{X} \rightarrow X$ be a modification such that $\widetilde{X}$ is smooth and $\pi^* s$ defines a normal crossings divisor. By introducing a partition of unity, $\Gamma_{\|\cdot\|}(\lambda)$ can be written as a finite sum of terms of the form
\[
    \frac{h(\lambda)}{\lambda^{2\kappa}}I(\lambda),
\]
where $h(\lambda)$ is given by \cref{eq:hlambda} and $I(\lambda)$ by \cref{eq:Ilambda}. Since $I(\lambda)/\lambda^{2\kappa}$ is holomorphic on $\mathfrak{Re}\,\lambda > 0$, by inspection of \cref{eq:hlambda}, we find that $(\lambda-p)^{2 \ell_p} \Gamma_{\|\cdot\|}(\lambda)$ is holomorphic in a neighborhood of $p$, where
\[
    \ell_p = \# \Big\{ (i,j) : i \in \{1,\hdots,\kappa\},\ j \in \{ 1,\hdots,N m_i -1\},\ \frac{j}{m_i} = p \Big\} \geq 1 \text{ for } p \in P_+.
\]
Thus, we have that
\begin{align*}
    \underset{\lambda = p}{\mathrm{Res}} \Big\{ \epsilon^{-\lambda} \lambda^{-1} \Gamma_{\|\cdot\|}(\lambda) \Big\} &= \underset{\lambda = p}{\mathrm{Res}} \Bigg\{ \epsilon^{-p} \sum\limits_{j=0}^{\infty} \frac{1}{j!}\big(\log\epsilon^{-1}\big)^j (\lambda-p)^{j-2\ell_p} \frac{(\lambda-p)^{2\ell_p}\Gamma_{\|\cdot\|}(\lambda)}{\lambda} \Bigg\} \\
    &= \underset{\lambda = p}{\mathrm{Res}} \Bigg\{ \epsilon^{-p} \sum\limits_{j=0}^{\infty} \frac{1}{j!}\big(\log\epsilon^{-1}\big)^j (\lambda-p)^{j-2\ell_p} \sum\limits_{k=0}^{\infty} c_k (\lambda-p)^k \Bigg\},
\end{align*}
where
\[
    c_k = \frac{1}{k!} \frac{\d^{k}}{\d \lambda^{k}} \bigg( \frac{(\lambda-p)^{2\ell_p} \Gamma_{\|\cdot\|}(\lambda)}{\lambda} \bigg) \bigg|_{\lambda = p}.
\]
We obtain
\[
    \underset{\lambda = p}{\mathrm{Res}}\Big\{ \epsilon^{-\lambda} \lambda^{-1} \Gamma_{\|\cdot\|}(\lambda) \Big\} = \epsilon^{-p} \sum_{j=0}^{2\ell_p - 1} \frac{1}{j!} \big( \log \epsilon^{-1}\big)^j c_{2\ell_p - 1 -j}.
\]
The coefficients $c_{2\ell_p - 1 -j}$ can be interpreted as the action of currents similar to $\mu^{\|\cdot\|}_j(\omega)$ on $\xi$.

\end{document}